\pdfoutput=1
\documentclass[12pt]{paper}

\def\figdir{.}

\input epsf.sty
\input labelfig.tex

\usepackage{color}
\usepackage{mathptmx}       
\usepackage{helvet}         
\usepackage{courier}        
\usepackage{type1cm}        
\usepackage{makeidx}         
\usepackage{graphicx}        
\usepackage{multicol}        
\usepackage[bottom]{footmisc}
\usepackage{amsmath}
\usepackage{amsfonts}
\usepackage{amssymb}
\usepackage{amsthm}
\usepackage{amsopn}
\usepackage{subcaption}
\usepackage{comment}
\newtheorem{theorem}{Theorem}
\newtheorem{corollary}[theorem]{Corollary}

\newtheorem{lemma}[theorem]{Lemma}

\newtheorem{proposition}[theorem]{Proposition}
\newtheorem{question}[theorem]{Question}
\newtheorem{remark}[theorem]{Remark}

\usepackage[top=2cm, bottom=2cm, left=2cm, right=2cm]{geometry}
\numberwithin{theorem}{section}
\numberwithin{figure}{section}
\numberwithin{equation}{section}
\DeclareMathOperator{\CR}{CR}

\DeclareMathOperator{\SLE}{SLE}
\DeclareMathOperator{\CLE}{CLE}

\DeclareMathOperator{\inrad}{inrad}

\DeclareMathOperator{\capa}{cap}
\DeclareMathOperator{\diam}{diam}

\DeclareMathOperator{\CGE}{CGE}

\usepackage[colorlinks=false,pdfborderstyle={/W 1}]{hyperref}

\def\lora{\longrightarrow}
\def\p{\partial}

\def\md{\mid}
\def\Bb#1#2{{\def\md{\bigm| }#1\bigl[#2\bigr]}}
\def\BB#1#2{{\def\md{\Bigm| }#1\Bigl[#2\Bigr]}}

\def\Pb{\Bb\P}

\def\PB{\BB\P}

\begin{document}

\title{A conformally invariant growth process of SLE excursions}
\author{G\'abor Pete and Hao Wu}


%
%
\maketitle

\abstract{We construct an aggregation process of chordal $\SLE_{\kappa}$ excursions in the unit disk, starting from the boundary, growing towards all inner points simultaneously, invariant under all conformal self-maps of the disk. We prove that this conformal growth process of excursions, abbreviated as $\CGE_\kappa$, exists if{f} $\kappa\in[0,4)$, and that it does not create additional fractalness: the Hausdorff dimension of the closure of all the $\SLE_\kappa$ arcs attached is $1+\kappa/8$ almost surely. We determine the dimension of points that are approached by $\CGE_\kappa$ at an atypical rate, and construct conformally invariant random fields on the disk based on $\CGE_\kappa$.}
\newcommand{\eps}{\epsilon}
\newcommand{\ov}{\overline}
\newcommand{\U}{\mathbb{U}}
\newcommand{\T}{\mathbb{T}}
\newcommand{\HH}{\mathbb{H}}
\newcommand{\LA}{\mathcal{A}}
\newcommand{\LC}{\mathcal{C}}
\newcommand{\LD}{\mathcal{D}}
\newcommand{\LF}{\mathcal{F}}
\newcommand{\LK}{\mathcal{K}}
\newcommand{\LE}{\mathcal{E}}
\newcommand{\LL}{\mathcal{L}}
\newcommand{\LU}{\mathcal{U}}
\newcommand{\LV}{\mathcal{V}}
\newcommand{\LZ}{\mathcal{Z}}
\newcommand{\LH}{\mathcal{H}}
\newcommand{\R}{\mathbb{R}}
\newcommand{\C}{\mathbb{C}}
\newcommand{\N}{\mathbb{N}}
\newcommand{\Z}{\mathbb{Z}}
\newcommand{\E}{\mathbb{E}}
\newcommand{\PP}{\mathbb{P}}
\renewcommand{\P}{\mathbb{P}}
\newcommand{\QQ}{\mathbb{Q}}
\newcommand{\A}{\mathbb{A}}
\newcommand{\bn}{\mathbf{n}}
\newcommand{\MR}{MR}
\newcommand{\cond}{\,|\,}
\newcommand{\bcond}{\,\big|\,}
\newcommand{\la}{\langle}
\newcommand{\ra}{\rangle}
\newcommand{\tree}{\Upsilon}

\section{Introduction}\label{s.intro}

Planar aggregation processes based on harmonic measure, usually called Laplacian growth models, have been extensively studied in the physics and mathematics literature, key examples being Diffusion Limited Aggregation and a family of models using iterated conformal maps, the Hastings-Levitov models; see \cite{witten1981diffusion,halsey2000diffusion,hastings1998laplacian,carleson2001aggregation,rohde2005some,norris2012hastings,viklund2015small}. The most interesting versions, which produce fractal growth according to simulations, are notoriously hard to study: the discrete ones do not have meaningful scaling limits, the continuum models do not have enough symmetries that would make their analysis easier. 

This motivated Itai Benjamini to suggest a model where both the building blocks and the aggregation measure are fully conformally  invariant. Firstly, there is a sigma-finite infinite measure on pairs of points on the boundary of the complex unit disk $\U$, unique up to a global constant factor, which is invariant under all conformal self-maps of the disk, the M\"obius transformations: it has density $H_{\U}(z,w)=c / |z-w|^2$, called the {\bf boundary Poisson kernel}. Secondly, once we have a pair of points $z,w \in \partial\U$, we can take a {\bf chordal $\SLE_\kappa$ arc} $\gamma$ in $\U$ between $z$ and $w$, with $\kappa\in[0,8)$. (For background on the Schramm-Loewner Evolution, see \cite{SchrammScalinglimitsLERWUST,WernerRandomPlanarcurves, LawlerConformallyInvariantProcesses}.) Then, we can take a point $z\in\U$ and a conformal uniformization map from the component of $\U\setminus\gamma$ that contains $z$ back to $\U$, normalized at $z$, and try and iterate this procedure. However, since our measure on $\partial\U\times \partial\U$ is not finite, we cannot take iid random pairs $(z_i,w_i)$ one after the other. Instead, we need to take a Poisson point process on $\partial\U\times \partial\U \times [0,\infty)$ with intensity measure $H_{\U}(z,w)\,dz\,dw\, dt$, take all the arrivals $\big\{(z_i,w_i) : i\in I_r(t)\big\}$ with time index in $[0,t)$ and arc-length larger than a small positive cutoff $r>0$, and do the above iterative procedure for these finitely many pairs of points. (See Figure~\ref{f.Arcs0} for an illustration.) Then, we let $r\to 0$, and hope that the process, using the increasing set $I_r(t)$ of arrivals, will converge to a process $(D^z_t, t \ge 0)$, the connected component of $z$ at time $t$. Moreover, using conformal invariance, we can try to define the process targeted at all points $z\in\U$ simultaneously: as long as $D^z_t=D^w_t$, the processes targeted towards $z$ and $w$ coincide, and after the disconnection time they continue independently. Our first result says that this envisioned procedure actually works, but only for $\kappa\in [0,4)$:

\begin{theorem}\label{thm::growing_process_conf_inv}
For $\kappa\in [0,4)$, there exists a growth process $(D_t, t\ge 0)=(D^z_t, t\ge 0, z\in\U)$ of $\SLE_{\kappa}$ excursions, targeted at all points, with the property that $(D_t, t\ge 0)$ and $(\varphi(D_t), t\ge 0)$ have the same law (with no time-change) for all M\"obius transformations $\varphi$ of $\U$. We will abbreviate this Conformal Growth of Excursions by $\CGE_\kappa$.
\end{theorem}

 \begin{figure}[htbp]
 \begin{center}
 \includegraphics[height=3.5 in]{\figdir/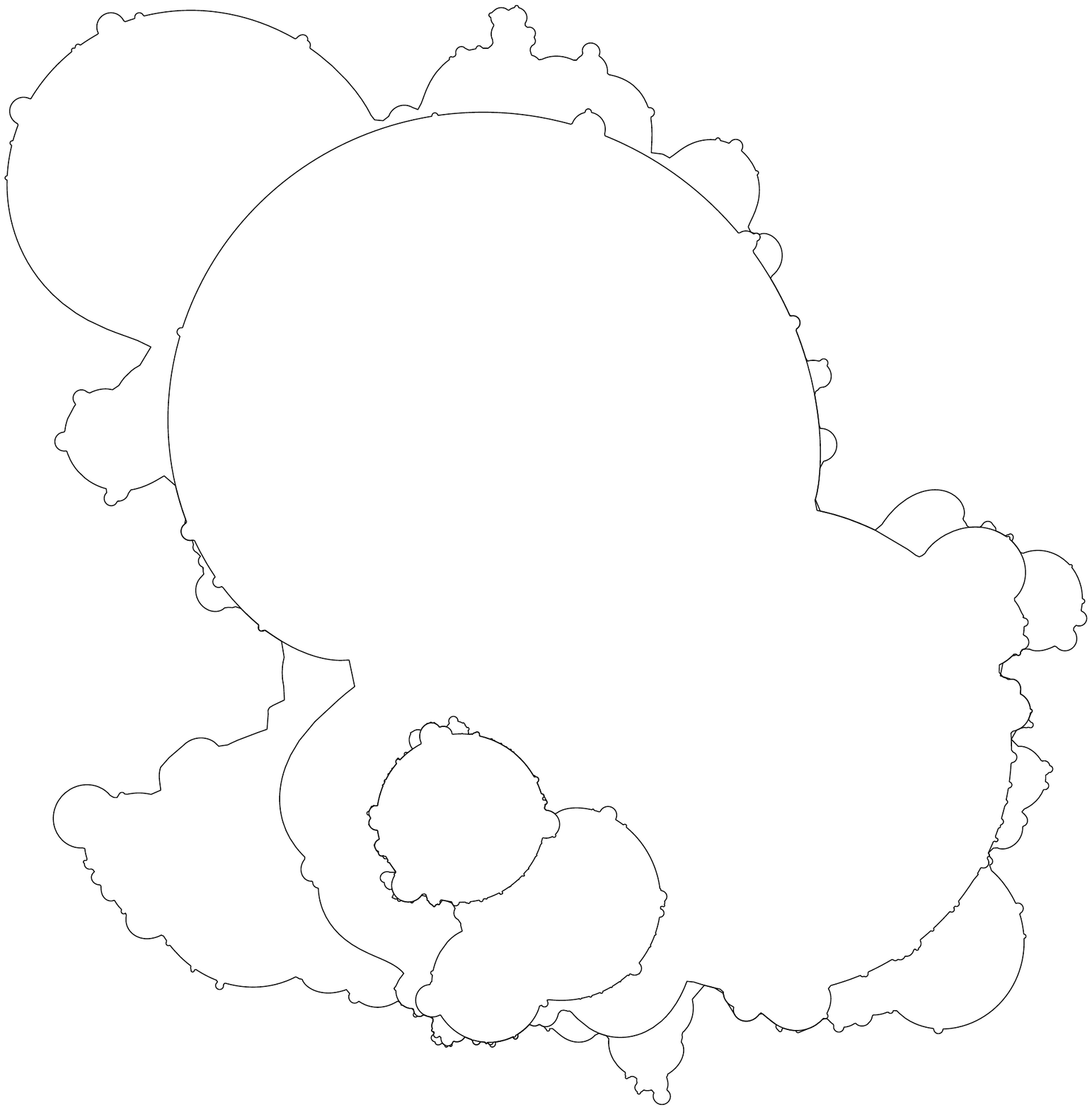}
 \end{center}
 \caption{Four stages (75, 150, 300, 500 arcs) of $\CGE_{\kappa=0}$, growing towards $\infty$ (that is, the process targeted at 0, inverted through the circle for better visibility), with some positive cutoff for the sizes of the $\SLE_0$ arcs, which are just semicircles.}
 \label{f.Arcs0}
 \end{figure}

Maybe disappointingly, $\CGE_\kappa$ does not produce considerable extra fractalness, beyond what is already inherent in the $\SLE_\kappa$ arcs, which have dimension $1+\kappa/8$ \cite{BeffaraDimension}:

\begin{theorem} \label{thm::growthprocess_dimension}
Fix $\kappa\in [0,4)$. Suppose that $(D_t^0, t\ge 0)$ is $\CGE_\kappa$ targeted at the origin. Define $\Gamma$ to be the closure of the union $\cup_{t\ge 0}\partial D_t^0$. Then, almost surely,
\[\dim(\Gamma)=1+\kappa/8.\]
\end{theorem}

Consider now the conformal radius of $D^z_t$ seen from $z$, denoted by $\CR(z;D^z_t)$. We can derive the asymptotic decay of the conformal radius. For $\lambda\in\R$, define the Laplace exponent 
\[\Lambda_{\kappa}(\lambda)=\log \E\left[\CR(z; D^z_1)^{-\lambda}\right].\]
As we will see, $\Lambda_{\kappa}(\lambda)$ is finite when $\lambda<1-\kappa/8$, and we have almost surely that 
\begin{equation}\label{eqn::cr_decay_normal}
\lim_{t\to\infty}\frac{-\log\CR(z; D^z_t)}{t}=\Lambda_{\kappa}'(0)\in (0,\infty)\,.
\end{equation}
From (\ref{eqn::cr_decay_normal}), we know that typically the conformal radius $\CR(z; D_t^z)$ decays like $\exp(-t\Lambda_{\kappa}'(0))$. We are also interested in those points $z$ where $\CR(z; D_t^z)$ decays in an abnormal way. Define, for $\alpha\ge 0$, the random set
\begin{equation}\label{eqn::Theta_alpha}
\Theta(\alpha)=\left\{z\in \U: \lim_{t\to \infty}\frac{-\log\CR(z; D_t^z)}{t}=\alpha\right\}.
\end{equation}
Clearly, when $\alpha\neq \Lambda'_{\kappa}(0)$, the points in the set $\Theta(\alpha)$ have an abnormal decaying rate of $\CR(z; D_t^z)$. The Hausdorff dimension of $\Theta(\alpha)$ can be estimated through \textit{Fenchel-Legendre transform} of $\Lambda_{\kappa}$.
The Fenchel-Legendre transform of $\Lambda_{\kappa}$ is defined by, for $\alpha\in\R$,
\[\Lambda_{\kappa}^*(\alpha)=\sup_{\lambda\in\R}\left(\lambda\alpha-\Lambda_{\kappa}(\lambda)\right).\] 

\begin{figure}[htbp]
\SetLabels
(.38*0.85){\color{blue}$\Lambda_\kappa(\lambda)$}\\
(.88*0.85){\color{blue}$\Lambda^*_\kappa(\alpha)$}\\
(.6*0.97){\color{red}$2\alpha$}\\
(.7*0.97){\color{blue}$(1-\kappa/8)\alpha$}\\
\endSetLabels
\centerline{
\AffixLabels{
\includegraphics[width=0.5\textwidth]{\figdir/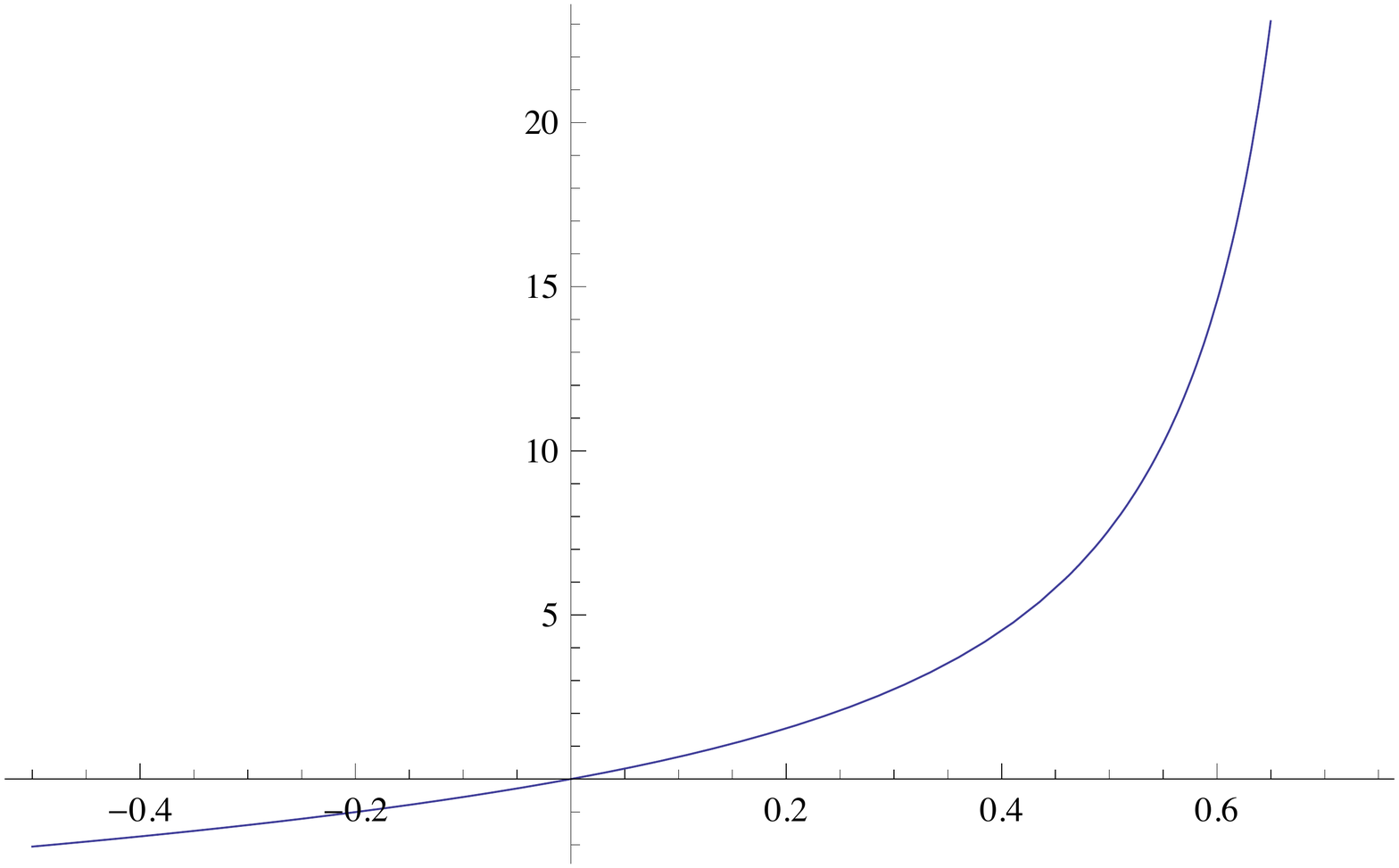}
\hskip 0.2 in
\includegraphics[width=0.5\textwidth]{\figdir/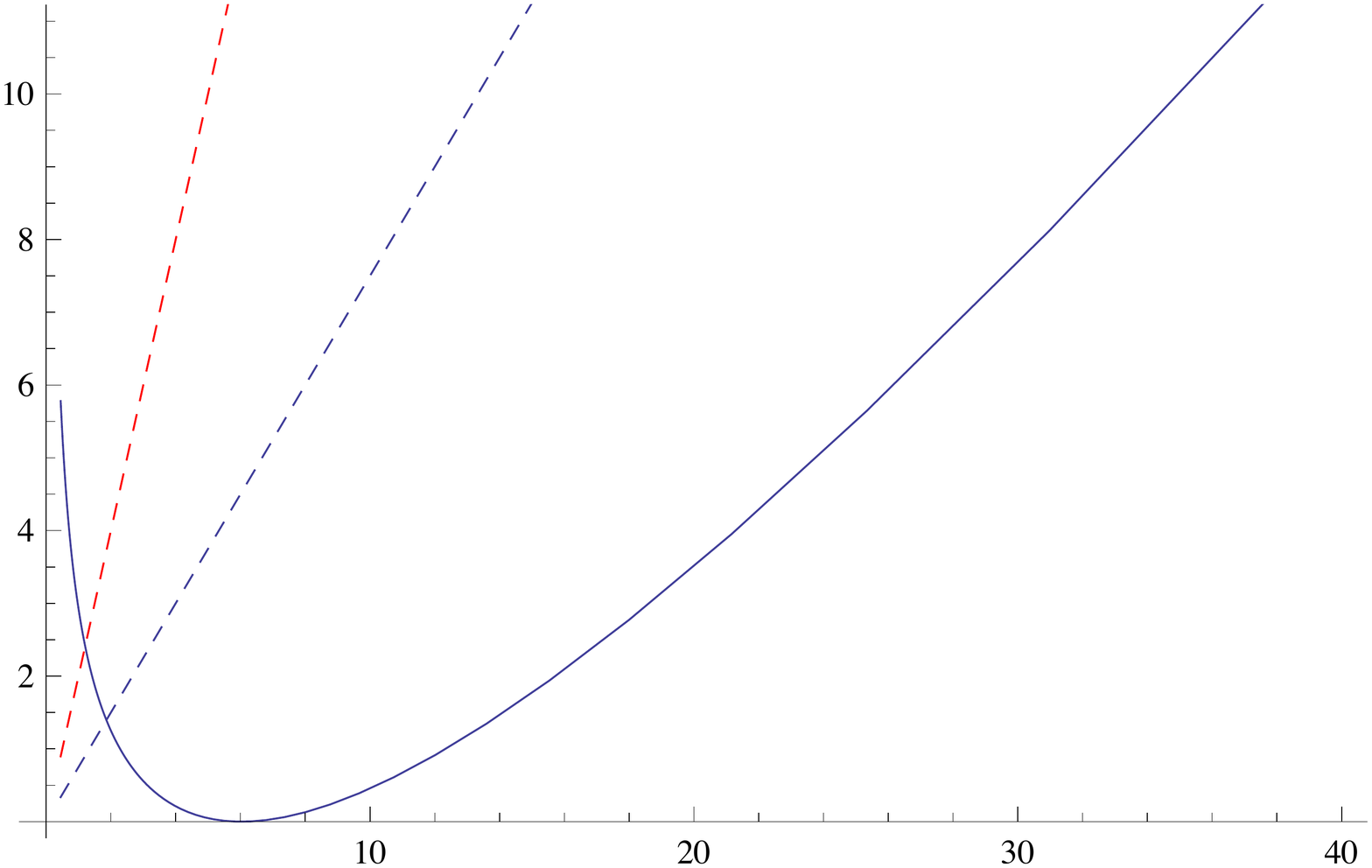}
}
}
\caption{With $\kappa=2$, numerical approximations of $\Lambda_\kappa(\lambda)$, blowing up at $\lambda=1-\kappa/8$, and its Fenchel-Legendre transform $\Lambda_\kappa^*(\lambda)$, with an asymptotic slope $1-\kappa/8$.}
\label{f.Lambda}
\end{figure}

\begin{theorem}\label{thm::abnormal_decay_upper}
Define 
\begin{equation}\label{eqn::alpha_min}
\alpha_{min}=\sup\{\alpha>0: 2\alpha-\Lambda_{\kappa}^*(\alpha)\le 0\}.
\end{equation}
We have almost surely,
\[\begin{cases}
\dim(\Theta(\alpha))\le 2-\Lambda^*_{\kappa}(\alpha)/\alpha, &  \alpha\ge \alpha_{min};\\
\Theta(\alpha)=\emptyset,& \alpha<\alpha_{min}.
\end{cases}\]
\end{theorem}

The $\CGE_\kappa$ process $(D_t, t\ge 0)$ targeted at all points naturally yields a fragmentation process of the unit disk, raising interesting questions. First of all, for any $z,w\in\U$, we can define $T(z,w)$ to be the first time $t$ such that $z,w$ are not in the same connected component of $D_t$. We call $T(z,w)$ the disconnection time, for which we have the following estimate:

\begin{theorem}\label{thm::disconnection_green}
Fix $\kappa\in [0,4)$ and let $\Lambda_{\kappa}(\lambda)$ be the Laplace exponent defined in (\ref{eqn::accumulated_capacity_laplace_exponent}). Let $z,w\in\U$ be distinct and $T(z,w)$ be the disconnection time of $\CGE_\kappa$ targeted at all points. Then there exists a constant $C\in (0,\infty)$ (only depending on $\kappa$) such that
\[\left|\E[T(z,w)]-G_{\U}(z,w)/\Lambda'_{\kappa}(0)\right|\le C,\] 
where $G_{\U}$ is Green's function of the unit disc.
\end{theorem}

Based on the fragmentation process, we can also construct random fields on the unit disk:

\begin{theorem}\label{thm::distribution_cvg}
Fix $\kappa\in [0,4)$ and $\delta>0$. Suppose that $\nu$ is a $\sigma$-finite measure supported in $[0,\infty)$ with finite mean and unit second moment:
\[\bar{m}:=\int x\nu[dx]<\infty,\quad \int x^2 \nu[dx]=1.\]
Suppose that $(D_t, t\ge 0)$ is $\CGE_\kappa$ targeted at all points, generated from $\Gamma$, the collection of $\SLE_{\kappa}$ excursions. For each $z\in\U$, denote by $D^z_t$ the connected component of $D_t$ that contains $z$. Given $\Gamma$, let $(\sigma_{\gamma})_{\gamma\in\Gamma}$ be i.i.d. weights sampled from $\nu$. Define, for each $z\in\U$, $t\ge 0$,
\[h_t(z)=\sum_{\gamma\in\Gamma}\sigma_{\gamma}1_{\{\gamma\text{ contributes to }D_t^z\}}-t\bar{m}.\] Then we have the following:
\begin{enumerate}
\item [(1)] There exists an $H^{-2-\delta}_{loc}(\U)$-valued random variable $h$ such that $h_t\to h$ as $t\to\infty$ almost surely in $H^{-2-\delta}_{loc}(\U)$. Moreover, for any $f,g\in C_c^{\infty}(\U)$, the covariance between $\langle h, f\rangle$ and $\langle h, g\rangle$ is given by 
\[\E[\langle h, f\rangle \langle h, g\rangle]=\iint_{\U\times\U}f(z)g(w)\E[T(z,w)]dzdw,\]
where $T(z,w)$ is the disconnection time of $\CGE_\kappa$.
\item [(2)] The limiting distribution $h$ is almost surely determined by $\Gamma$ and $(\sigma_{\gamma})_{\gamma\in\Gamma}$.
\item [(3)] The limiting distribution is conformal invariant: let $\varphi$ be any M\"obius transformation of $\U$, and let $\tilde{h}$ be the limiting distribution in $H^{-2-\delta}_{loc}(\U)$ associated with $\tilde{\Gamma}:=(\tilde{\gamma}=\varphi(\gamma), \gamma\in\Gamma)$ and $(\sigma_{\varphi^{-1}(\tilde{\gamma})})$, then $\tilde{h}=h\circ\varphi^{-1}$ almost surely.
\end{enumerate}
\end{theorem}

Finally, let us discuss the most obvious question: in what discrete models can one find a structure that has our $\CGE_\kappa$ as a scaling limit? The full conformal invariance of the process targeted at all points suggests that probably one should look for structures that can be defined not only as growth processes, but also as static objects, similar to the Conformal Loop Ensembles $\CLE_\kappa$ \cite{SheffieldWernerCLE,WernerWuCLEExploration,MillerWatsonWilsonCLENestingfield}; note however that $\CLE_\kappa$ exists for a different subset of $\kappa$ values: for $\kappa\in (8/3,4]$ if the loops are simple, and for $\kappa\in (8/3,8)$ in general. Also, from such a discrete static point of view, it would seem to be an important property that the order of attaching the $\SLE_\kappa$ arcs should not matter, and such a commutation relation holds only for $\kappa=2$ \cite{Kozdron2007configurational}. Therefore, a good candidate for a suitable discrete model is {\bf Wilson's algorithm} \cite{wilson1996generating}, which generates a Uniform Spanning Tree from iteratively adding Loop-Erased Random Walk trajectories, which converge to $\SLE_2$ arcs \cite{SchrammScalinglimitsLERWUST,LawlerSchrammWernerLERWUST}. Another discrete model that has many features similar to our $\CGE_\kappa$ could be the construction of a {\bf critical percolation} configuration \cite{SmirnovPercolationConformalInvariance, CamiaNewmanPercolationFull, WernerLecturePercolation} from the collection of outermost crossings, viewed from an inner point. However, these crossings converge not to $\SLE_{8/3}$, but different $\SLE_{8/3}(\rho)$ arcs \cite{LawlerSchrammWernerConformalRestriction,WernerConformalRestrictionRelated}, hence we do not have an exact correspondence.

\bigskip

\noindent\textbf{Overview of the paper.} In Section~\ref{s.prelim}, we define the boundary Poisson kernel and the $\SLE_\kappa$ and $\SLE_{\kappa}(\rho)$ processes, prove an overshoot estimate for subordinators, and recall a sufficient condition for the convergence of random fields in the sense of distributions.

In Section~\ref{s.construction}, we prove Theorem~\ref{thm::growing_process_conf_inv}: we construct the growth process for $\kappa<4$, prove that it is conformally invariant, and show that it does not exist for $\kappa\geq 4$. The proofs are based on known estimates on the probability that chordal $\SLE_\kappa$ comes close to a point on the boundary or inside the domain, and on conformal martingales related to these questions, describing the Laplace transform of the capacity of an $\SLE_\kappa$ arc. We also prove Theorem~\ref{thm::abnormal_decay_upper} on the dimension of points with abnormal decay, using Large Deviations theory.

In Section~\ref{s.field}, we give estimates on the disconnection time, proving Theorem~\ref{thm::disconnection_green} in particular. Here a key ingredient is the innocent-looking Lemma~\ref{lem::key_lemma}, saying that the process leaves the boundary $\p\U$ in finite time with positive probability. We also prove Theorem~\ref{thm::growthprocess_dimension} on the dimension being $1+\kappa/8$, where the full conformal invariance of the process targeted at all points is of immense help. Finally, we prove Theorem~\ref{thm::distribution_cvg}, constructing the conformally invariant random fields.

We end the paper with three open problems in Section~\ref{s.open}.

\bigskip

\noindent\textbf{Acknowledgments.} We are indebted to Itai Benjamini for suggesting the model back in 2006, and to Oded Schramm for his insights at the beginning of this project. We also thank St\'ephane Benoist, Tom Ellis, Adrien Kassel, Gregory Lawler and Wendelin Werner for helpful discussions. 
G. Pete is partially supported by the Hungarian National Science Fund OTKA grant K109684,
and by the MTA R\'enyi Institute ``Lend\"ulet'' Limits of Structures Research Group.
H. Wu is supported by the NCCR/SwissMAP, the ERC AG COMPASP, and the Swiss NSF.

\section{Preliminaries}\label{s.prelim}

\noindent\textbf{Notation.}
\[B(z,r)=\{w\in\C: |z-w|<r\},\quad \U=B(0,1); \quad \HH=\{w\in\C: \Im{w}>0\}.\]
\subsection{Green function and Poisson kernel}\label{sec::poisson_kernel}

Let $\eta(z, \cdot; t)$ denote the law of 2D Brownian motion $(B_s, 0\le s\le t)$ starting from $z$. We can write 
\[\eta(z,\cdot; t)=\int_{\C}\eta(z,w; t)dw,\]
where $dw$ denotes the area measure and $\eta(z, w; t)$ is a measure on continuous curve from $z$ to $w$. Define $\eta(z,w)=\int_0^{\infty}\eta(z,w;t)dt$. This is an infinite $\sigma$-finite measure. If $D$ is a domain and $z,w\in D$, define $\eta_D(z,w)$ to be $\eta(z,w)$ restricted to curves stayed in $D$. If $z\neq w$ and $D$ is a domain such that a Brownian motion in $D$ eventually exits $D$, then the total mass $|\eta_D(z,w)|$ is finite and we define the \textit{Green function}
\[G_D(z,w)=\pi|\eta_D(z,w)|.\]
In particular, when $D=\U$ and $z,w\in\U$, we have 
\[G_{\U}(z,w)=\log \left|\frac{1-\bar{z}w}{z-w}\right|.\] 
Just like planar Brownian motion, the Green function is also conformally invariant: if $\varphi : D \lora \varphi(D)$ is a conformal map and $z,w\in D$, then we have 
\begin{equation}\label{eqn::greenfunction_conf_inv}
G_{\varphi(D)}(\varphi(z), \varphi(w))=G_D(z,w).
\end{equation}

Suppose that $D$ is a connected domain with piecewise analytic boundary. Let $B$ be a Brownian motion starting from $z\in D$ and stopped at the first exit time $\tau_D:=\inf\{t: B_t\not\in D\}$. Denote by $\eta_D(z,\partial D)$ the law of $(B_t, 0\le t\le\tau_D)$. We can write 
\[\eta_D(z,\partial D)=\int_{\partial D} \eta_D(z,w)dw,\]
where $dw$ is the length measure on $\partial D$ and $\eta_D(z,w)$ is a measure on continuous curves from $z$ to $w$. The measure $\eta_D(z,w)$ can be viewed as a measure on Brownian paths starting from $z$ and restricted to to exit $D$ at $w$. Define \textit{Poisson kernel} $H_D(z,w)$ to be the total mass of $\eta_D(z,w)$. 

Suppose that $z,w$ are distinct boundary points. Define the measure on Brownian paths from $z$ to $w$ in $D$ to be
\[\eta_D(z,w)=\lim_{\eps\to 0}\frac{1}{\eps}\eta_D(z+\eps \bold{n}_z, w),\]
where $\bold{n}_z$  is the inward normal at $z$. The measure $\eta_D(z,w)$ is called Brownian excursion measure. Define the\textit{boundary Poisson kernel} $H_D(z,w)$ to be the total mass of $\eta_D(z,w)$. From the conformal invariance of Brownian motion, we can derive the conformal covariance of the boundary Poisson kernel (see \cite[Proposition 5.5]{LawlerConformallyInvariantProcesses}): Suppose that $\varphi : D\lora \varphi(D)$ is a conformal map and $z,w\in\partial D$ and $\varphi(z), \varphi(w)\in\partial \varphi(D)$ are analytic boundary points, then 
\begin{equation}\label{eqn::poisson_kernel_conf_cov}
|\varphi'(z)\varphi'(w)|H_{\varphi(D)}(\varphi(z), \varphi(w))=H_D(z,w).
\end{equation} 
Moreover, if $D=\U$ and $z,w\in\partial \U$, we have 
\[H_{\U}(z,w)=\frac{1}{\pi|z-w|^2}.\]
In particular, if $\theta=\arg(z/w)\in[0,2\pi)$, we have 
\begin{equation}\label{eqn::poisson_kernel_estimate}
H_{\U}(z,w)=\frac{1}{4\pi \sin^2(\theta/2)}.
\end{equation}

\subsection{Chordal and radial SLE}\label{sec::sle}

In this section, we review briefly the chordal and radial $\SLE_{\kappa}(\rho)$ processes and refer the reader to \cite{WernerRandomPlanarcurves} and  \cite{LawlerConformallyInvariantProcesses} for a detailed introduction. The \textit{chordal Loewner chain} with a continuous driving function $W: [0,\infty)\to \R$ is the solution for the following ODE: for $z\in \overline{\HH}$,
\[\partial_t g_t(z)=\frac{2}{g_t(z)-W_t},\quad g_0(z)=z.\]
This solution is well-defined up to the swallowing time 
\[T(z):=\inf\big\{t: \inf_{s\in [0,t]}|g_s(z)-W_s|>0\big\}.\]
For $t\ge 0$, define $K_t:=\{z\in\overline{\HH}: T(z)\le t\}$, then $g_t(\cdot)$ is the unique conformal map from $\HH\setminus K_t$ onto $\HH$ with the expansion $g_t(z)=z+2t/z+o(1/z)$ as $z\to\infty$. 

\textit{Chordal $\SLE_{\kappa}$} is the chordal Loewner chain with driving function $W=\sqrt{\kappa}B$ where $B$ is a one-dimensional Brownian motion. For $\kappa\in [0,4]$, the $\SLE_{\kappa}$ process is almost surely a continuous simple curve in $\HH$ from $0$ to $\infty$.  
Suppose $\gamma$ is an $\SLE_{\kappa}$ curve in $\HH$ from $0$ to $\infty$, then it is conformal invariant: for any $c>0$, the curve $c\gamma$ has the same law as $\gamma$ (up to time change). Therefore, we could define chordal $\SLE$ in any simply connected domain. Suppose that $D$ is a simply connected domain and $x,y\in\partial D$ are distinct boundary points. Define $\SLE_{\kappa}$ in $D$ from $x$ to $y$ to be the image of $\SLE_{\kappa}$ in $\HH$ from $0$ to $\infty$ under any conformal map from $\HH$ onto $D$ sending the pair $(0,\infty)$ to $(x,y)$. 

Chordal $\SLE_{\kappa}$ is reversible: suppose that $\gamma$ is an $\SLE_{\kappa}$ in $D$ from $x$ to $y$, then the time-reversal of $\gamma$ has the same law as an $\SLE_{\kappa}$ in $D$ from $y$ to $x$. Thus, we also call  $\SLE_{\kappa}$ in $D$ from $x$ to $y$ as $\SLE_{\kappa}$ in $D$ with two end points $(x,y)$. 

The \textit{radial Loewner chain} with a continuous driving function $W: [0,\infty)\to\partial\U$ is the solution for the following ODE: for $z\in\overline{\U}$,
\[\partial_t g_t(z)=g_t(z)\frac{W_t+g_t(z)}{W_t-g_t(z)},\quad g_0(z)=z.\]
This solution is well-defined up to the swallowing time 
\[T(z):=\inf\big\{t: \inf_{s\in [0,t]}|g_s(z)-W_s|>0\big\}.\]
For $t\ge 0$, define $K_t:=\{z\in\overline{\U}: T(z)\le t\}$, then $g_t(\cdot)$ is the unique conformal map from $\U\setminus K_t$ onto $\U$ with the normalization: $g_t(0)=0, g_t'(0)>0$. 

\textit{Radial $\SLE_{\kappa}$} is the radial Loewner chain with driving function $W=\exp(i\sqrt{\kappa}B)$ where $B$ is a one-dimensional Brownian motion. For $\kappa\in [0,8)$, radial $\SLE_{\kappa}$ is almost surely a continuous curve in $\U$ from $1$ to the origin. \textit{Radial $\SLE_{\kappa}(\rho)$} with $W_0=x\in\partial\U$ and force point $V_0=y\in\partial \U$ is the radial Loewner chain with driving function $W$ solving the following SDEs:
\begin{align*}
dW_t&=i\sqrt{\kappa}B_t-\left(\frac{\kappa}{2}W_t+\frac{\rho}{2}W_t\frac{W_t+V_t}{W_t-V_t}\right)dt,
\quad W_0=x;\\
dV_t&=V_t\frac{W_t+V_t}{W_t-V_t}dt, \quad V_0=y.
\end{align*}
The system has a unique solution up to the collision time 
\[T:=\inf\{t: W_t=V_t\}.\]

We focus on the weight $\rho=\kappa-6$ for the following reason: chordal $\SLE_{\kappa}$ in $\U$ from $x$ to $y$ has the same law as radial $\SLE_{\kappa}(\kappa-6)$  with starting point $W_0=x$ and force point $V_0=y$; see \cite{SchrammWilsonSLECoordinatechanges}. Fix $\kappa\in [0,8)$ and $\rho=\kappa-6$. Define $\theta_t=\arg(W_t)-\arg(V_t)\in (0,2\pi)$, then by It\^{o}'s formula, the process $\theta_t$ satisfies the SDE:
\begin{equation}\label{eqn::radial_sle_theta}
d\theta_t=\sqrt{\kappa}dB_t+\frac{\kappa-4}{2}\cot(\theta_t/2)dt. 
\end{equation}
The collision time $T$ is also the first time that $\theta_t$ hits $0$ or $2\pi$. Moreover, when $\kappa\in [0,8)$, we have $\E[T]<\infty$.

Suppose that $D$ is a proper simply connected domain. The \textit{conformal radius} of $D$ seen from $z\in D$ is $|\varphi'(z)|^{-1}$ where $\varphi$ is any conformal map from $D$ onto $\U$ sending $z$ to the origin. We denote this conformal radius by $\CR(z;D)$. 
 Define the inradius 
 \[\inrad(z;D):=\inf_{w\in \C\setminus D}|z-w|.\]
 By Koebe's one quarter theorem and the Schwarz lemma \cite[Theorem 3.17, Lemma 2.1]{LawlerConformallyInvariantProcesses}, we have that 
 \begin{equation}\label{eqn::koebe_quarter}
 \inrad(z;D)\le \CR(z;D)\le 4\inrad(z;D).
 \end{equation}
 For any compact subset $K\subset\overline{\U}$, let $D$ be the connected component of $\U\setminus K$ that contains $z$. Define the \textit{capacity} of $K$ seen from $z$ to be 
 \[\capa(z;K)=-\log\CR(z;D).\]
 When $z$ is the origin, we simply denote $\CR(0;D)$ and $\capa(0;K)$ by $\CR(D)$ and $\capa(K)$ respectively. One can check that the radial Loewner chain is parameterized by capacity seen from the origin.

\subsection{Overshoot estimate for subordinators}

Suppose that $(X(t), t\ge 0)$ is a right-continuous increasing process starting from 0 and taking values in $[0,\infty)$. We call $X$ a subordinator if it has independent homogeneous increments on $[0,\infty)$. The Laplace transform of a subordinator has a nice expression: for $t>0$ and $\lambda\ge 0$, we have
$\E[\exp(\lambda X(t))]=\exp(-t\Phi(\lambda))$, 
where $\Phi: [0,\infty)\to [0,\infty)$. There exist a unique constant $d\ge 0$ and a unique measure $\Pi$ on $(0,\infty)$, which is called the L\'evy measure of $X$, with $\int (1\wedge x)\Pi(dx)<\infty$ such that, for $\lambda\ge 0$, 
\[\Phi(\lambda)=d\lambda+\int(1-e^{-\lambda x})\Pi(dx).\]
Moreover, one has almost surely, for $t>0$, 
\[X(t)=dt+\sum_{s\le t}\Delta_s,\]
where $(\Delta_s, s\ge 0)$ is a Poisson point process with intensity $\Pi$. (More precisely, we have a Poisson point process $\{(\Delta_j,s_j) : j\in J\}$ with intensity $\Pi\otimes \mathsf{Lebesgue}$ on $(0,\infty) \times [0,\infty)$, the second coordinate being time, where $J$ is a countable set, and we let $\Delta_s:=\Delta_j$ when $s=s_j$, while $\Delta_s:=0$ otherwise.) Define the tail of the L\'evy measure:
\[\overline{\Pi}(x)=\Pi((x,\infty)).\]

We introduce the inverse of $X$: for $x>0$, 
\[L_x=\inf\{t: X(t)>x\},\]
and the processes of first-passage and last-passage of $X$: for $x>0$,
\[D_x=X(L_x),\quad G_x=X(L_x-).\]

A subordinator is a transient Markov process, its potential measure $U(dx)$ is called the renewal measure. It is defined as, for any nonnegative measurable function $f$,
\[\int_0^{\infty}f(x)U(dx)=\E\left[\int_0^{\infty}f(X(t))dt\right].\]

\begin{proposition}\label{prop::firstlast_density}
For every real numbers $a,b,x$ such that $0\le a <x\le a+b$, we have that
\[\PP[G_x\in da, D_x-G_x\in db]=U(da)\Pi(db).\]
\end{proposition}
\begin{proof}
\cite[Lemma 1.10]{BertoinSubordinators}.
\end{proof}

\begin{proposition}\label{prop::overshoot}
Suppose that $X$ is a subordinator with L\'evy measure $\Pi$ satisfying
\[\int(e^{\lambda_0x}-1)\Pi(dx)<\infty,\quad\text{for some }\lambda_0>0.\]
Then there exists a positive finite constant $C$ (depending on $\Pi$ and $\lambda_0$) such that
\[\PP[D_x-x\ge y]\le Ce^{-\lambda_0y},\quad\text{for all }x\ge 0,y\ge 0.\]
\end{proposition}
\begin{proof}
When $y\in [0,1]$, we could take $C\ge e^{\lambda_0}$. Thus we can suppose $y\ge 1$. We divide the event $\{D_x-x\ge y\}$ according to the values of $G_x$. For every $k\in\N$, define
\[E_k=[G_x\le x-k,D_x\ge x+y].\]
By Proposition \ref{prop::firstlast_density}, we have that
\begin{align*}
\PP[E_k]&\le\overline{\Pi}(k+y)\\
&\le \int_{u\ge k+y}e^{\lambda_0 u}e^{-\lambda_0(k+y)}\Pi(du)\\
&\le e^{-\lambda_0(k+y)}\int_{u\ge 1}e^{\lambda_0 u}\Pi(du).
\end{align*}
Thus
\[\PP[D_x-x\ge y]\le\sum_{k\ge 0}\PP[E_k]\le e^{-\lambda_0 y}\frac{\int_{u\ge 1}e^{\lambda_0 u}\Pi(du)}{1-e^{-\lambda_0}}.\]
So we can take
\[C=e^{\lambda_0}\vee \frac{\int_{u\ge 1}e^{\lambda_0 u}\Pi(du)}{1-e^{-\lambda_0}}.\]
\end{proof}

\begin{remark}
In the literature, people usually consider right-continuous subordinators. Whereas, the conclusions in this section also hold for left-continuous subordinators. Note that if $X$ is a left-continuous subordinator, then $X$ can be written as, for $t>0$,
\[X(t)=dt+\sum_{s<t}\Delta_s,\]
where $(\Delta_s, s\ge 0)$ is a Poisson point process. Therefore, the proofs in this section can be modified for left-continuous subordinators without difficulty. In the later part of our paper, we will apply conclusions in this section for left-continuous subordinators. 
\end{remark}

\subsection{Convergence of distributions}

In this subsection, we give an overview of the convergence of distributions. We refer the reader to \cite[Section 1.13]{TaoEpsilon} for a more detailed introduction to the space of distributions and to \cite[Section 5]{MillerWatsonWilsonCLENestingfield} for the proof of the convergence result of distributions that we will need.

Fix a positive integer $d$. The \textit{Schwartz space} $S(\R^d)$ is defined to be the space of smooth functions $f : \R^d\lora \C$ such that all derivatives are rapidly decreasing, i.e., $|x|^n \|\nabla^j f(x)\|$ is bounded for all non-negative integers $n$ and $j$ where we view $\nabla^j f(x)$ as a $d^j$-dimensional vector. We equip $S(\R^d)$ with the topology generated by the family of seminorms $\|f\|_{n,j}:=\sup_{x\in\R^d}|x|^n \|\nabla^j f(x)\|$  for $n, j\ge 0$. 

The space $S(\R^d)^*$ of \textit{tempered distributions} is defined to be the space of continuous linear functionals on $S(\R^d)$. We write the evaluation of $h\in S(\R^d)^*$ on $f\in S(\R^d)$ with the notation $\la h, f\ra$. For $h\in S(\R^d)^*$, we say that $h$ is supported in a closed set $K$ if $\la h, f\ra=0$ for all $f\in S(\R^d)$ that vanish on an open neighborhood of $K$. The support of $h$ is the intersection of all $K$ that $h$ is supported on. For $h\in S(\R^d)^*$ and $f\in C_c^{\infty}(\R^d)$, define the product $hf\in S(\R^d)^*$ by $\la hf, g\ra=\la h, \bar{f}g\ra$ for all $g\in S(\R^d)$. 

For $f\in S(\R^d)$, the Fourier transform and its conjugate are defined by \[\LF(f)(\xi)=\int e^{-2\pi i x\cdot \xi} f(x) dx,\quad \LF^*(f)(\xi)=\int e^{2\pi i x\cdot \xi} f(x) dx,\quad \forall \xi\in\R^d.\]
Note that $f\in S(\R^d)$ implies $\LF(f)\in S(\R^d)$. We may define the Fourier transform of a tempered distribution $h$ by $\la \LF(h), f\ra :=\la h, \LF^*(f)\ra$ for all $f\in S(\R^d)$. We also denote $\LF(h)$ by $\hat{h}$.

For any $s\in\R$, define the \textit{Sobolev space} $H^s(\R^d)$ to be the space of all tempered distributions $h$ such that the distribution $\la \xi\ra^s \hat{h}(\xi)$ lies in $L^2(\R^d)$ where $\la \xi\ra:=(1+|\xi|^2)^{1/2}$. The space $H^s(\R^d)$ is a Hilbert space with the inner product
\[\la f, g\ra_{H^s(\R^d)}:=\int \la\xi\ra^{2s} \hat{f}(\xi)\overline{\hat{g}(\xi)}d\xi.\]

Let $D\subset \R^d$ be an open set. For $h\in C_c^{\infty}(\R^d)^*$, we say that $h\in H_{loc}^s (D)$ if $hf\in H^s(D)$ for all $f\in C_c^{\infty}(D)$. We equip $H^s_{loc}(D)$ with a topology generated by the seminorms $\|f\cdot\|_{H^s(\R^d)}$, which implies that $h_n\to h$ in $H^s_{loc}(D)$ if and only if $\langle h_n, f\rangle \to \langle h, f\rangle$ in $H^s(\R^d)$ for all $f\in C_c^{\infty}(D)$.

\begin{proposition}\label{prop::distribution_cvg_criterion}
Let $D\subset \R^d$ be an open set and $\delta>0$. Suppose that $(h_n)_{n\in \N}$ is a sequence of random measurable functions defined on $D$ satisfying the following conditions: for every compact $K\subset D$,
\begin{enumerate}
\item [(1)] for all $n\ge 1$, we have \[\int _K \E[|h_n(z)|^2]dz<\infty;\]
\item [(2)] there exists a summable sequence $(a_n)_{n\in\N}$ of positive reals such that, for all $n\ge 1$, we have  
\[\iint_{K\times K}\left|\E\left[(h_{n+1}(z)-h_n(z))(h_{n+1}(w)-h_n(w))\right]\right|dzdw\le a_n^3.\]
\end{enumerate} 
Then there exists a random element $h\in H_{loc}^{-d-\delta}(D)$ supported on the closure of $D$ such that almost surely $h_n\to h$ in $H_{loc}^{-d-\delta}(D)$. 
\end{proposition}
\begin{proof}
\cite[Proposition 5.1]{MillerWatsonWilsonCLENestingfield}.
\end{proof}

\section{Construction of the growth process $\CGE_\kappa$}\label{s.construction}

\subsection{The Poisson point process of SLE excursions}
\label{sec::growthprocess_construction}

Let $H_{\U}(x,y)$ be the boundary Poisson kernel for the unit disc $\U$ with distinct boundary points $x,y\in\U$ as introduced in Section \ref{sec::poisson_kernel}. Denote by $\mu_{\U,\kappa}^{\#}(x,y)$ the law of chordal $\SLE_{\kappa}$ in $\U$ with two end points $x, y$. For $\kappa\in [0,8)$, define the \textit{$\SLE_{\kappa}$ excursion measure} to be 
\begin{equation}\label{eqn::sle_excursion_measure}
\mu_{\U, \kappa}=\int\int dxdy H_{\U}(x,y)\mu^{\#}_{\U,\kappa}(x,y),
\end{equation} 
where $dx, dy$ are length measures on $\partial\U$. Note that $\mu_{\U, \kappa}$ is an infinite $\sigma$-finite measure. From the conformal invariance of SLE and the conformal covariance of boundary Poisson kernel (\ref{eqn::poisson_kernel_conf_cov}), we can derive the conformal invariance of SLE excursion measure.

\begin{proposition}\label{prop::sle_excursion_measure_conf_inv}
The $\SLE$ excursion measure $\mu_{\U, \kappa}$ is conformal invariant: for any M\"obius transformation $\varphi$ of $\U$, we have 
\[\varphi\circ \mu_{\U, \kappa}=\mu_{\U, \kappa},\]
where $\varphi\circ\mu[A]:=\mu[\gamma: \varphi(\gamma)\in A]$.
\end{proposition}

We will construct a growth process from a Poisson point process of $\SLE$ excursions. The construction is not surprising if one is familiar with \cite{SheffieldWernerCLE} and \cite{WernerWuCLEExploration}. To be self-contained, we briefly explain the construction. Let $(\gamma_t, t\ge 0)$ be a Poisson point process with intensity $\mu_{\U,\kappa}$. More precisely, let $((\gamma_j, t_j), j\in J)$ be a Poisson point process with intensity $\mu_{\U, \kappa}\otimes[0,\infty)$, and then arrange the excursions $\gamma_j$ according to $t_j$: denote the excursion $\gamma_j$ by $\gamma_t$ if $t=t_j$ and $\gamma_t$ is empty set if there is no $t_j$ that equals $t$. There are only countably many excursions in $(\gamma_t, t\ge 0)$ that are not empty set. For $\kappa\in [0,8)$, with probability one there is no $\gamma_t$ passing through the origin. For each $t$ such that $\gamma_t$ is not the empty set, the curve $\gamma_t$ separates $\U$ into two connected components, and we denote by $U_t^0$ the one that contains the origin. Let $f_t$ be the conformal map from $U_t^0$ onto $\U$ normalized at the origin: $f_t(0)=0, f_t'(0)>0$. For $t>0$, define the accumulated capacity to be 
\[X_t=\sum_{s<t}\capa(\gamma_s).\]
\begin{proposition}\label{prop::accumulated_capacity_finite}
For $t>0$, the accumulated capacity $X_t$ is almost surely finite if and only if $\kappa\in [0,4)$. 
\end{proposition} 
We will complete the proof of Proposition \ref{prop::accumulated_capacity_finite} in Section \ref{sec::accumulated_capacity_laplace}. Assuming Proposition \ref{prop::accumulated_capacity_finite}, we can now construct the growth process for $\kappa\in [0,4)$. For any fixed $T>0$ and $r>0$, let $t_1(r)<t_2(r)<\cdots<t_j(r)$ be the times $t$ before $T$ at which the distance between the two end points of $\gamma_t$ is at least $r$. Define 
\[\Psi^r_T=f_{t_j(r)}\circ\cdots\circ f_{t_1(r)}.\]
The map $\Psi^r_T$ is a conformal map from some subset of $\U$ onto $\U$. 
By Proposition \ref{prop::accumulated_capacity_finite}, we know that $X_T<\infty$ almost surely when $\kappa\in [0,4)$. Then the conformal map $\Psi^r_T$ converges almost surely in the Carath\'eodory topology seen from the origin, as $r\to 0$; see \cite[Section 4.3, Stability of Loewner chains]{SheffieldWernerCLE}. Define 
\[\left(D_t^0:=\Psi_t^{-1}(\U), t\ge 0\right).\]
This is a decreasing sequence of simply connected domains containing the origin, and we call it \textit{the growth process $\CGE_\kappa$ of $\SLE$ excursions targeted at the origin}. By the conformal invariance of the $\SLE$ excursion measure, we can derive the conformal invariance of $\CGE_\kappa$.

\begin{lemma} \label{lem::growing_process_conf_inv}
For $\kappa\in [0,4)$, the law of the growth process $(D_t^0, t\ge 0)$ is conformally invariant under any M\"obius transformation $\varphi$ of $\U$ that preserves the origin. 
\end{lemma}

\begin{proof}
Let $(\hat{\gamma}_t, t\ge 0)$ be a Poisson point process with intensity $\mu_{\U, \kappa}$, and define $\hat{f}_t$ and $\hat{\Psi}_t$ for each $t$ as described above, and denote by $(\hat{D}_t^0, t\ge 0)$ the corresponding growth process targeted at the origin. 

By Proposition \ref{prop::sle_excursion_measure_conf_inv}, we know that the process $(\gamma_t:=\varphi(\hat{\gamma}_t), t\ge 0)$ is also a Poisson point process with intensity $\mu_{\U, \kappa}$. Define $f_t$ and $\Psi_t$ for each $t$, and denote by $(D_t^0, t\ge 0)$ be the corresponding growth process targeted at the origin. It is clear that 
\[f_t=\varphi\circ\hat{f}_t\circ\varphi^{-1},\quad \Psi_t=\circ_{s<t}f_s=\varphi\circ\hat{\Psi}_t\circ\varphi^{-1}.\] 
Since $(\gamma_t, t\ge 0)$ has the same law as $(\hat{\gamma}_t, t\ge 0)$, the process $(D_t^0=\varphi(\hat{D}_t^0), t\ge 0)$ has the same law as $(\hat{D}_t^0, t\ge 0)$ as desired.
\end{proof}

We can construct $\CGE_\kappa$ targeted at any $z\in\U$ in the same way as above, except that we choose to normalize at $z$ instead of normalizing at the origin. Another way to describe $\CGE_\kappa$ targeted at $z$ would be $(\varphi(D_t^0), t\ge 0)$ where $\varphi$ is any M\"obius transformation of $\U$ that sends the origin to $z$. By Lemma \ref{lem::growing_process_conf_inv}, the choice of $\varphi$ does not affect the law of $(\varphi(D_t^0), t\ge 0)$, thus $\CGE_\kappa$ targeted at $z$ is well-defined.

Now we will describe the relation between two growth processes targeted at distinct points $z, w\in\U$. Let $(D_t^z, t\ge 0)$ (resp. $(D_t^w, t\ge 0)$) be $\CGE_\kappa$ processes targeted at $z$ (resp. targeted at $w$), and define $T(z,w)$ (resp. $T(w,z)$) to be the first time $t$ that $w\not\in D_t^z$ (resp. $z\not\in D_t^w$). We call $T(z,w)$ the disconnection time. The interesting property of these growth processes is that the two processes have the same law up to the disconnection time. 

\begin{proposition}\label{prop::relation_two_growing_processes}
For $\kappa\in [0,4)$, and for any $z, w\in \U$, the law of $(D_t^z, t<T(z,w))$ is the same as the law of $(D^w_t, t<T(w,z))$.
\end{proposition}  
\begin{proof}
Recall a classical result about Poisson point processes (see \cite[Section 0.5]{BertoinLevyProcesses}): Let $(a_t, t\ge 0)$ be a Poisson point process with some intensity $\nu$ (defined in some metric space $A$). Let $\LF_{t-}=\sigma(a_s, s<t)$. If $(\Phi_t, t\ge 0)$ is a process (with values on functions of $A$ to $A$) such that for any $t>0$, $\Phi_t$ is $\LF_{t-}$-measurable, and that $\Phi_t$ preserves $\nu$, then $(\Phi_t(a_t), t\ge 0)$ is still a Poisson point process with intensity $\nu$.

Let $(\hat{\gamma}_t, t\ge 0)$ be a Poisson point process with intensity $\mu_{\U, \kappa}$ and define $\LF_{t-}=\sigma(\hat{\gamma}_s, s<t)$, let $\hat{f}_t^z$ and $\hat{\Psi}_t^z$ be the conformal maps as described above normalized at $z$. Let $(\hat{D}^z_t, t\ge 0)$ be the corresponding growth process targeted at $z$ and $\hat{T}(z,w)$ be the first time $t$ that $w\not\in \hat{D}_t^z$. For each $t<\hat{T}(z,w)$, the domain $\hat{D}_t^z$ contains $w$, and let $G_t$ be the conformal map from $\hat{D}_t^z$ onto $\U$ normalized at $w$: $G_t(w)=w$ and $G_t'(w)>0$. 

For each $t<\hat{T}(z,w)$, define $\varphi_t=G_t\circ(\hat{\Psi}^z_t)^{-1}$. For $t=\hat{T}(z,w)$, define $\varphi_t=\lim_{s\uparrow t}\varphi_s$. For $t>\hat{T}(z,w)$, define $\varphi_t$ to be identity map. Note that $\hat{\Psi}^z_t=\circ_{s<t}\hat{f}_s^z$ and thus $\hat{\Psi}^z_t$ is $\LF_{t-}$-measurable. Therefore, for all $t>0$, $\varphi_t$ is $\LF_{t-}$-measurable. By Proposition \ref{prop::sle_excursion_measure_conf_inv} and the classical result of Poisson point process recalled at the beginning of the proof, we know that $(\gamma_t:=\varphi_t(\hat{\gamma}_t), t\ge 0)$ is also a Poisson point process with intensity $\mu_{\U, \kappa}$. For $(\gamma_t, t\ge 0)$, let $(D_t^w, t\ge 0)$ be the corresponding growth process targeted at $w$ and let $T(w,z)$ be the first time $t$ that $z\not\in D_t^w$. By the construction, we have that 
\[D^w_t=\hat{D}^z_t, \quad \text{for all }t<T(w,z).\]
Hence, in this coupling, we have $T(w,z)\le \hat{T}(z,w)$. By symmetry, we have that $T(w,z)= \hat{T}(z,w)$ almost surely. In particular, this coupling implies that the two disconnection times have the same law, and the two growth processes $(D^z_t, t\ge 0)$ and $(D^w_t, t\ge 0)$ have the same law up to the disconnection time. 
\end{proof}

Proposition \ref{prop::relation_two_growing_processes} tells that, for any $z,w\in\U$, it is possible to couple the two growth processes targeted at $z$ and $w$ respectively to be identical up to the first time at which the points $z,w$ are disconnected. Hence, it is possible to couple the growth processes $(D^z_t, t\ge 0)$ for all $z$ in a fixed countable dense subset of $\U$ simultaneously in such a way that for any two points $z$ and $w$, the above statement holds. 

For such a coupling, we get a Markov process on domains $(D_t, t\ge 0)$: At $t=0$, the domain is $\U$, and at time $t>0$, it is the union of all the disjoint open subsets corresponding to the growth process targeted at all points $z$ at time $t$. We call this Markov process \textit{the conformal growth process of $\SLE$ excursions targeted at all points}, or $\CGE_\kappa$. By construction, it is naturally conformal invariant. This completes the proof of Theorem \ref{thm::growing_process_conf_inv}.

In Subsections \ref{sec::chordal_sle_laplace} and \ref{sec::accumulated_capacity_laplace}, we will calculate the Laplace transform of the accumulated capacity and complete the proof of Proposition~\ref{prop::accumulated_capacity_finite}. 

\subsection{The Laplace transform of the capacity of chordal SLE}\label{sec::chordal_sle_laplace}

 In this section, we will calculate the Laplace transform of the capacity of chordal $\SLE_{\kappa}$ in $\U$. To this end, we need to recall some basic facts about hypergeometric functions. The hypergeometric function is defined for $|z|<1$ by the power series
\begin{equation}\label{eqn::hypergeometric_function}
\phi(a,b;c;z)=\sum_{n=0}^\infty \frac{(a)_n (b)_n}{(c)_n n!}z^n,
\end{equation}
where $(q)_n$ is the Pochhammer symbol defined by 
\[(q)_n=\begin{cases}
1,&n=0;\\
q(q+1)\cdots(q+n-1), &n\ge 1.
\end{cases}\]
The function in (\ref{eqn::hypergeometric_function}) is only well-defined for $c\not\in\{0,-1,-2,-3...\}$. The hypergeometric function is a solution of Euler's hypergeometric differential equation
\begin{equation}\label{eqn::hypergeometric_ode}
-ab\phi+(c-(a+b+1)z)\phi'+z(1-z)\phi''=0.
\end{equation}
Note that, when $c\not\in \Z$, the following function is also a solution to Equation (\ref{eqn::hypergeometric_ode}):
\[z^{1-c}\phi(1+a-c, 1+b-c; 2-c; z).\]

\begin{proposition}\label{prop::sle_capacity_laplace}
Fix $\kappa\in [0,8)$ and $\lambda\in (0,1-\kappa/8)$. 
Suppose that $\gamma=\gamma^\theta$ is a chordal $\SLE_{\kappa}$ in $\U$ from $x\in\partial\U$ to $y\in\partial\U$,  where $\theta=\arg(x)-\arg(y)$. Define three constants 
\[a=1-\frac{4}{\kappa}+\sqrt{\left(1-\frac{4}{\kappa}\right)^2+\frac{8\lambda}{\kappa}},\quad b=1-\frac{4}{\kappa}-\sqrt{\left(1-\frac{4}{\kappa}\right)^2+\frac{8\lambda}{\kappa}},\quad c=\frac{3}{2}-\frac{4}{\kappa}.\]
Assume that $c\not\in\Z$. Define two functions $f$ and $g$: for $u\in [0,1]$, 
\[f(u)=\phi(a,b;c;u),\quad g(u)=u^{1-c}\phi(1+a-c,1+b-c;2-c;u).\]
Then, we have
\begin{equation}\label{eqn::sle_capacity_laplace}
\E[\exp(\lambda \capa(\gamma^\theta))]=f(u)+\frac{(1-f(1))}{g(1)}g(u),
\end{equation}
where $u=\sin^2(\theta/4)$.
\end{proposition}

\begin{proof}
First, let us check the values of the functions $f$ and $g$ at the end points $u=0$ or $u=1$. Since $\kappa\in [0,8)$, $\lambda \in (0,1-\kappa/8)$ and $c\not\in\Z$, we have that 
\[a\in(0,1),\quad b\in (1-8/\kappa, 0),\quad c\in (-\infty, 1)\setminus\Z.\]
Combining with \cite[Page 104, Equation (46)]{emot}, we have that (denoting by $\Gamma$ the Gamma function)
\begin{align*}
f(0)=1,\quad f(1)&=\frac{\Gamma(c)\Gamma(c-a-b)}{\Gamma(c-a)\Gamma(c-b)}=\frac{\cos\left(\pi\sqrt{\left(1-\frac{4}{\kappa}\right)^2+\frac{8\lambda}{\kappa}}\right)}{\cos\left(\pi \left(1-\frac{4}{\kappa}\right)\right)}\in [-1, 1);\\
g(0)=0,\quad g(1)&=\frac{\Gamma(2-c)\Gamma(1-c)}{\Gamma(1-a)\Gamma(1-b)}\in (0,\infty).
\end{align*}
\smallbreak
Second, assuming the same notation as in Subsection \ref{sec::sle}, we know that $\gamma$ has the same law as radial $\SLE_{\kappa}(\kappa-6)$ with $(W_0, V_0)=(x, y)$. Let $\theta_t=\arg(W_t)-\arg(V_t)$, 
we will argue that $e^{\lambda t}f(\sin^2(\theta_t/4))$ and $e^{\lambda t}g(\sin^2(\theta_t/4))$ are martingales up to the collision time $T$.  
Suppose that $F$ is an analytic function defined on $(0,1)$. By the SDE (\ref{eqn::radial_sle_theta}) and It\^o's formula, we know that $e^{\lambda t}F(\sin^2(\theta_t/4))$ is a local martingale if and only if 
\[\lambda F(u)+\frac{3\kappa-8}{16}(1-2u)F'(u)+\frac{\kappa}{8}u(1-u)F''(u)=0.\]
Since $f$ and $g$ are solutions to this ODE, we know that $e^{\lambda t}f(\sin^2(\theta_t/4))$ and $e^{\lambda t}g(\sin^2(\theta_t/4))$ are local martingales. Since $f$ and $g$ are finite at end points $u=0$ and $u=1$, and the collision time $T$ has finite expectation, we may conclude that the processes $e^{\lambda t}f(\sin^2(\theta_t/4))$ and $e^{\lambda t}g(\sin^2(\theta_t/4))$ are martingales up to $T$.
\smallbreak
Finally, we derive Equation (\ref{eqn::sle_capacity_laplace}). Since $e^{\lambda t}f(\sin^2(\theta_t/4))$ and $e^{\lambda t}g(\sin^2(\theta_t/4))$ are martingales up to $T$ which has finite expectation, Optional Stopping Theorem gives 
\begin{align*}
\E\left[\exp(\lambda T)1_{\{\theta_T=0\}}\right]+&f(1)\E\left[\exp(\lambda T)1_{\{\theta_T=2\pi\}}\right]=f(u);\\
&g(1)\E\left[\exp(\lambda T)1_{\{\theta_T=2\pi\}}\right]=g(u).
\end{align*}
These give Equation (\ref{eqn::sle_capacity_laplace}) by noting that $\capa(\gamma)=T$.
\end{proof}

\begin{remark} The martingale $e^{\lambda t}f(\sin^2(\theta_t/4))$ in the proof of Proposition \ref{prop::sle_capacity_laplace} was studied in \cite{SchrammSheffieldWilsonConformalRadii}. 
\end{remark}

\begin{remark} \label{rem::sle_continuity}
For $\kappa\in [0,8)$, suppose that $\gamma$ is a chordal $\SLE_{\kappa}$ in $\U$ with distinct end points $x,y\in\partial \U$. For $\lambda\in\R$, define 
\[F(\kappa,\lambda)=\E[\exp(\lambda\capa(\gamma))].\]
On the one hand, by Proposition \ref{prop::sle_capacity_laplace}, we see that when $c=3/2-4/\kappa\not\in\Z$ and $\lambda<1-\kappa/8$, the quantity $F(\kappa,\lambda)$ is finite. On the other hand, we know that $F(\kappa,\lambda)$ is continuous in $(\kappa,\lambda)$ (for the continuity in $\kappa$; see for instance \cite[Theorem 1.10]{KemppainenSmirnovRandomCurves}). Therefore, the quantity $F(\kappa,\lambda)$ is finite for all $\kappa\in [0,8)$ and $\lambda<1-\kappa/8$.
\end{remark}

Since the boundary Poisson kernel $H_{\U}(x,y)$ blows up when $\theta=\arg(x)-\arg(y)$ is small, in order to understand the excursion measure $\mu_{\U, \kappa}$, it will be important for us how the capacity $\capa(\gamma^\theta)$ behaves as $\theta\to 0$. 

\begin{proposition}\label{prop::cap_small_u}
Fix $\kappa\in [0,8)$ such that $c(\kappa)=\frac{3}{2}-\frac{4}{\kappa} \not\in\Z$, and $\lambda\in (0,1-\kappa/8)$. As $\theta\to 0$, or equivalently, $u\to 0$, the Laplace transform of the capacity satisfies
\begin{equation}\label{eq::small_u_Lap}
\E\big[\exp(\lambda\capa(\gamma^\theta))-1\big]\asymp 
\left\{
\begin{tabular}{lll}
$u^{1-c}$&$\asymp \theta^{2(1-c)}$ &\text{if}$\quad 8/3 < \kappa < 8$\,,\\
$u$ &$\asymp \theta^2$  &\text{if}$\quad 0< \kappa < 8/3,\  c\not\in\Z_-$\,,
\end{tabular}
\right.
\end{equation}
with the constant factors implicit in $\asymp$ depending on $\kappa$ and $\lambda$.
Quite similarly, the expectation itself satisfies
\begin{equation}\label{eq::small_u_E}
\E\big[\capa(\gamma^\theta)\big]\asymp 
\left\{
\begin{tabular}{lll}
$u^{1-c}$ &$\asymp \theta^{2(1-c)}$ &if $\quad 8/3 < \kappa < 8$\,,\\
$u\log (1/u)$ &$\asymp \theta^2\log(1/\theta)$  &\text{if}$\quad \kappa =8/3$\,,\\
$u$ &$\asymp \theta^2$  &\text{if}$\quad 0\leq \kappa < 8/3$\,,
\end{tabular}
\right.
\end{equation}
with the constant factors implicit in $\asymp$ depending on $\kappa$.
\end{proposition}

\begin{proof}
For the functions given in Proposition~\ref{prop::sle_capacity_laplace} for the case $c\not\in\Z$, it is easy to check that 
$f(u)-1=f(u)-f(0)$ decays like $u$ as $u\to 0$ and $g(u)$ decays like $u^{1-c}$ as $u\to 0$. This gives~(\ref{eq::small_u_Lap}), with a phase transition at $c=0$, that is, at $\kappa=8/3$. 

To get the expectation~(\ref{eq::small_u_E}), one possibility would be to take the derivative of the Laplace transform $F(\kappa,\lambda)$ at $\lambda=0$. However, our formula~(\ref{eqn::sle_capacity_laplace}) is not particularly simple, hence this task is not  obvious. Another approach could be to argue that, for small $\theta$, it is very likely that $\capa(\gamma^\theta)$ is also small,  hence $\exp(\lambda\capa(\gamma^\theta))-1$ and $\lambda\capa(\gamma^\theta)$ are likely to be close to each other, and hence it is not surprising if the $u\to 0$ asymptotics of their expectations are the same. However, a large portion of the expectations might come from when $\capa(\gamma^\theta)$ is large, hence this argument would also need some extra work. Finally, we have~(\ref{eqn::sle_capacity_laplace}) and~(\ref{eq::small_u_Lap}) only when $c(\kappa)\not\in\Z$, which is an immediate drawback to start with. Therefore, we give the following separate and direct argument.

For $\kappa\in (0,8)$ and $\theta < r < 1/4$, it follows immediately from the results of \cite{AlbertsKozdronIntersectionProbaSLEBoundary} that
\begin{equation}\label{eq::AKdiam}
\P[ \diam(\gamma^\theta) > r ] \asymp (\theta/r)^{8/\kappa-1}.
\end{equation}
When the diameter is in $[r,2r)$, the capacity is at most $C r^2$. Moreover, if a curve $\gamma^\theta$ going from $\exp(-i\theta/2)$ to $\exp(i\theta/2)$ has diameter in $[r,2r)$, and it also separates the center $0$ from the point $1-r/2$, then its  capacity is at least $c r^2$. For $\SLE_\kappa$ from $\exp(-i\theta/2)$ to $\exp(i\theta/2)$, conditioned to have diameter in $[r,2r)$, this separating event has a uniformly positive probability, hence~(\ref{eq::AKdiam}) implies that
\begin{equation}\label{eq::AKcap}
\E\big[ \capa(\gamma^\theta) 1_{\left\{\diam(\gamma^\theta) \in [2^k\theta,2^{k+1}\theta)\right\}} \big] \asymp (2^k\theta)^2 (2^{-k})^{8/\kappa-1} = \theta^2 (2^k)^{3-8/\kappa}.
\end{equation}
Summing this up over dyadic scales from around $\theta$ to around $1/4$, we get that 
\begin{equation*}
\E\big[ \capa(\gamma^\theta) 1_{\left\{\capa(\gamma^\theta) < 1\right\}} \big] \asymp 
\begin{cases} 
\theta^{8/\kappa-1} &\text{if}\quad \kappa > 8/3,\\
\theta^2\log(1/\theta) & \text{if}\quad \kappa=8/3,\\
\theta^2 &\text{if}\quad \kappa < 8/3.
\end{cases}
\end{equation*}
Note that the last line holds  even for $\kappa=0$.

Now, $\capa(\gamma^\theta)$ is larger than $t \geq 1$ only if $\gamma^\theta$ comes closer than $\exp(-t)$ to 0. The probability of this has an exponential tail by the one-point estimate in the dimension upper bound \cite[Lemma 6.3, Theorem 8.1]{RohdeSchrammSLEBasicProperty}, so this part of the probability space does not raise the expectation $\E[\capa(\gamma^\theta)]$ by more than a factor, and the proof of~(\ref{eq::small_u_E}) is complete.
\end{proof}

\subsection{The Laplace transform of the accumulated capacity}\label{sec::accumulated_capacity_laplace}

For $\kappa\in [0,8)$, recall that $\mu_{\U, \kappa}$ is the SLE excursion measure defined in (\ref{eqn::sle_excursion_measure}). Let $(\gamma_t, t\ge 0)$ be a PPP with intensity $\mu_{\U,\kappa}$ and define the accumulated capacity in the same way as before:
\[X_t=\sum_{s<t}\capa(\gamma_s).\]
By Campbell's formula, we have that, for $\lambda\in\R$, 
\begin{equation}\label{eqn::accumulated_capacity_laplace}
\E[\exp(\lambda X_t)]=\exp\left(t\int(e^{\lambda\capa(\gamma)}-1)\mu_{\U,\kappa}[d\gamma]\right).
\end{equation}
In particular, the left hand side of (\ref{eqn::accumulated_capacity_laplace}) is finite if and only if the right hand side is finite. We will study the Laplace exponent
\begin{equation}\label{eqn::accumulated_capacity_laplace_exponent}
\Lambda_{\kappa}(\lambda):=\int(e^{\lambda\capa(\gamma)}-1)\mu_{\U,\kappa}[d\gamma]=\int\int dxdy H_{\U}(x,y)\mu_{\U, \kappa}^{\#}(x,y)\left[e^{\lambda\capa(\gamma)}-1\right].
\end{equation}

\begin{proposition}\label{prop::accumulated_capacity_laplace_exponent}\ 
\begin{enumerate}
\item[(1)] 
When $\kappa\in [0,4)$, the Laplace exponent $\Lambda_{\kappa}(\lambda)$ is finite for $\lambda\in (0,1-\kappa/8)$ and infinite for $\lambda\geq 1-\kappa/8$. If $\kappa\geq 4$, then $\Lambda_{\kappa}(\lambda)$ is infinite for all $\lambda>0$.
\item[(2)] When $\kappa\in [0,4)$, we have that $\E[X_t]<\infty$ for all $t$. When $\kappa\ge 4$, we have that $\E[X_t]=\infty$ for all $t>0$.
\end{enumerate}
\end{proposition}

\begin{proof}
First, we show that $\Lambda_{\kappa}(\lambda)$ is finite when $\kappa\in [0,4)$, $\lambda\in (0,1-\kappa/8)$, and  $c:=3/2-4/\kappa\not\in\Z$. Note that, in (\ref{eqn::accumulated_capacity_laplace_exponent}), the boundary Poisson kernel and the expectation $\mu_{\U, \kappa}^{\#}(x,y)\left[e^{\lambda\capa(\gamma)}-1\right]$ only depend on the angle difference of $x,y$. Assuming the same notation as in Proposition~\ref{prop::sle_capacity_laplace}, we see that 
\begin{align*}
\Lambda_{\kappa}(\lambda)&=4\pi\int_0^{\pi}\frac{d\theta}{4\pi \sin^2(\theta/2)}\mu_{\U, \kappa}^{\#}(1, e^{i\theta})\left[e^{\lambda\capa(\gamma)}-1\right]\tag{by (\ref{eqn::poisson_kernel_estimate})}\\
&=\int_0^{\pi}\frac{d\theta}{\sin^2(\theta/2)}\left(f(\sin^2(\theta/4))-1+\frac{1-f(1)}{g(1)}g(\sin^2(\theta/4))\right)\tag{by (\ref{eqn::sle_capacity_laplace})}\\
&=\frac{1}{2}\int_0^{1/2}du\ u^{-3/2}(1-u)^{-3/2}\left(f(u)-1+\frac{1-f(1)}{g(1)}g(u)\right)\tag{set $u=\sin^2(\theta/4)$}.
\end{align*}
Using the decay rate (\ref{eq::small_u_Lap}) of Proposition~\ref{prop::cap_small_u}, the exponent $\Lambda_{\kappa}(\lambda)$ is finite for $\lambda\in (0,1-\kappa/8)$ when $c<1/2$ which is to say $\kappa<4$, and infinite for every $\lambda>0$ when $\kappa\geq 4$.

That is, $\Lambda_{\kappa}(\lambda)$ is finite  for $\kappa\in [0,4)\setminus\{8/3, 8/5, 8/7, ....\}$ and $\lambda\in (0,1-\kappa/8)$. It is infinite when $\lambda\ge 1-\kappa/8$, since already the integrand, the right hand side of (\ref{eqn::sle_capacity_laplace}),  explodes.

To extend this for $\kappa\in\{8/3, 8/5, 8/7, ....\}$, the continuity of $\kappa\mapsto F^\theta(\kappa,\lambda)$ mentioned in Remark \ref{rem::sle_continuity} implies that the singularity in the integrand can be bounded from above by something integrable when $\lambda < 1-\kappa/8$, and from below by something non-integrable when $\lambda\ge 1-\kappa/8$.

For part (2) regarding $\E[X_t]$, we can do the analogous calculation, just using the decay rate~(\ref{eq::small_u_E}) instead of~(\ref{eq::small_u_Lap}).
\end{proof}

\begin{proof}[Proof of Proposition \ref{prop::accumulated_capacity_finite}]
When $\kappa\in [0,4)$, we proved in Proposition \ref{prop::accumulated_capacity_laplace_exponent} that $\E[X_t]$ is finite, and thus the accumulated capacity $X_t$ is finite almost surely.

Next, we will argue that $X_t$ diverges almost surely when $t>0, \kappa\ge 4$. For $k\ge 1$, define $M_k$ to be the number of excursions $\gamma_s$ with $s<t$ such that $2^{-k}\le \capa(\gamma_s)<2^{-k+1}$. Since $(\gamma_s, s\ge 0)$ is a PPP with intensity $\mu_{\U, \kappa}$, we know that $M_k$ is a Poisson random variable with parameter
\[q_k:=t\mu_{\U, \kappa}[\gamma: 2^{-k}\le \capa(\gamma)<2^{-k+1}].\]
By part (2) of Proposition~\ref{prop::accumulated_capacity_laplace_exponent} , we have $\E[X_t]=\infty$ for $\kappa\ge 4$, thus
\[\sum_{k\ge 1}2^{-k}q_k\ge \E[X_t]/2=\infty.\]
Since $(M_k, k\ge 1)$ are independent Poisson random variables, we have 
\[\E\left[\exp\left(-\sum_{k\ge 1}2^{-k}M_k\right)\right]=\Pi_{k\ge 1}\E\left[\exp\left(-2^{-k}M_k\right)\right]=\Pi_{k\ge 1}\exp\left(-q_k\left(1-e^{-2^{-k}}\right)\right)=0.\]
Therefore, when $\kappa\ge 4$, we almost surely have
\[X_t\ge \sum_{k\ge 1}2^{-k}M_k=\infty.\]
\end{proof}

\subsection{Extremes of the conformal radii}

Fix $\kappa\in [0,4)$, by Proposition \ref{prop::accumulated_capacity_laplace_exponent}, we know that the Laplace exponent $\Lambda_{\kappa}(\lambda)$ is finite for $\lambda\in (-\infty, 1-\kappa/8)$. In particular, this implies that it is differentiable on $(-\infty, 1-\kappa/8)$ (see \cite[Lemma 2.2.5]{DemboZeitouniLargeDeviations}). Moreover, by Strong Law of Large Numbers for subordinators (\cite[Page 92]{BertoinLevyProcesses}), we have almost surely that
\[\lim_{t\to\infty}\frac{X_t}{t}=\Lambda_{\kappa}'(0)\in (0,\infty),\]
which implies (\ref{eqn::cr_decay_normal}). To prove Theorem  \ref{thm::abnormal_decay_upper}, we first summarize some basic properties of $\Lambda_{\kappa}(\lambda)$ and $\Lambda_{\kappa}^*(x)$ (see \cite[Lemmas 2.2.5, 2.2.20]{DemboZeitouniLargeDeviations}, and our Figure~\ref{f.Lambda} in the Introduction):
\begin{enumerate}
\item [(1)] The Laplace exponent $\Lambda_{\kappa}(\lambda)$ is convex and smooth on $(-\infty, 1-\kappa/8)$.
\item [(2)] The Fenchel-Legendre transform $\Lambda_{\kappa}^*(x)$ is non-negative, convex, and smooth on $(0,\infty)$.
\item [(3)] We have $\Lambda_{\kappa}^*(x)=0$ when $x=\Lambda_{\kappa}'(0)$, the function $\Lambda_{\kappa}^*$ is increasing on $(\Lambda_{\kappa}'(0), \infty)$ and is decreasing on $(0, \Lambda_{\kappa}'(0))$.
\item [(4)] Since $\Lambda_{\kappa}(\lambda)\to-\infty$ as $\lambda\to-\infty$, we know that 
\[\Lambda_{\kappa}^*(0)=+\infty,\quad \Lambda_{\kappa}^*(x)\uparrow+\infty \text{ as }x\downarrow 0.\]
\item [(5)] As $x\to +\infty$, we have 
\[\lim_{x\to\infty}\frac{\Lambda_{\kappa}^*(x)}{x}=1-\kappa/8.\]
\end{enumerate}

Recall that $\alpha_{min}$ is defined through (\ref{eqn::alpha_min}). From the above properties, we know that $2x-\Lambda_{\kappa}^*(x)=0$ has a unique solution which is equal to $\alpha_{min}\in (0, \Lambda_{\kappa}'(0))$. We can complete the proof of Theorem~\ref{thm::abnormal_decay_upper} using the theory of large deviations:
 
\begin{proof}[Proof of Theorem \ref{thm::abnormal_decay_upper}]
It suffices to give upper bound for $\Theta(\alpha)\cap B(0, 1-\delta)$ for any $\delta>0$. Fix $\alpha\ge 0$ and assume $\beta>\alpha$. For $n\ge 1$, let $\LU_n$ be the collection of open balls with centers in $e^{-n\beta}\Z^2\cap B(0, 1-\delta/2)$ and radius $e^{-n\beta }$. For each ball $U\in \LU_n$, denote by $z(U)$ the center of $U$. Define, for $u^-<u^+$,
\[\LU_n(u^-, u^+)=\left\{U\in\LU_n: u^-\le\frac{-\log \CR\left(z(U); D_n^{z(U)}\right)}{n}\le u^+\right\}.\]
By Cram\'er's theorem (see \cite[Theorem 2.2.3]{DemboZeitouniLargeDeviations}), for $u^-<u^+$, for any $U\in\LU_n$, we have 
\begin{align}\label{eqn::large_deviation}
\PP[U\in\LU_n(u^-, u^+)]&=\PP\left[u^-n\le -\log \CR\left(z(U); D_n^{z(U)}\right)\le u^+n\right]\notag\\
&\le \exp\left(-n\left(\inf_{u^-\le u\le u^+}\Lambda_{\kappa}^*(u)+o(1)\right)\right),
\end{align}
where the $o(1)$ term tends to zero as $n\to\infty$ uniformly in $U$. 
Define 
\[\LC_m(u^-, u^+)=\cup_{n\ge m}\LU_n(u^-, u^+).\]
We claim that $\LC_m(\alpha^-, \alpha^+)$ is a cover for $\Theta(\alpha)\cap B(0,1-\delta)$ for any $\alpha^-<\alpha<\alpha^+<\beta$ and any $m\ge 1$. Pick $\tilde{\alpha}^-\in (\alpha^-, \alpha), \tilde{\alpha}^+\in (\alpha, \alpha^+)$. For any $z\in\Theta(\alpha)\cap B(0,1-\delta)$, since $\lim (-\log \CR(z; D^z_n))/n =\alpha$, we have that, for $n$ large enough,
\[\exp(-n\tilde{\alpha}^-)\ge \CR(z; D_n^z)\ge \exp(-n\tilde{\alpha}^+).\]
Let $w$ be the point in $e^{-n\beta}\Z^2$ that is the closest to $z$ and denote by $U$ the ball in $\LU_n$ with center $w$. 
Since $\tilde{\alpha}^+<\beta$ and by (\ref{eqn::koebe_quarter}), we know that $w$ is contained in $D_n^z$. Moreover, for $n$ large enough, by (\ref{eqn::koebe_quarter}) and that $\beta>\alpha^+>\tilde{\alpha}^+>\tilde{\alpha}^->\alpha^-$, we have 
\[
\CR(w;D_n^w)\ge \inrad(w; D_n)\ge \inrad(z; D_n)-e^{-n\beta}\ge \frac{1}{4}\CR(z; D_n)-e^{-n\beta}\ge \frac{1}{4}e^{-n\tilde{\alpha}^+ }-e^{-n\beta}\ge e^{-n\alpha^+}.\]
\[
\CR(w;D_n^w)\le 4\inrad(w; D_n)\le 4(\inrad(z; D_n)+e^{-n\beta})\le 4(\CR(z; D_n)+e^{-n\beta})\le 4(e^{-n\tilde{\alpha}^- }+e^{-n\beta})\le e^{-n\alpha^-}.\]
Therefore $z\in U\in\LU_n(\alpha^-, \alpha^+)$. This implies that $\LC_m(\alpha^-, \alpha^+)$ is a cover for $\Theta(\alpha)\cap B(0, 1-\delta)$. 
We use these covers to bound $s$-Hausdorff measure of $\Theta(\alpha)\cap B(0, 1-\delta)$. For $m\ge 1$, and $\alpha^-<\alpha<\alpha^+<\beta$, we have 
\begin{align*}
\E[\LH_s(\Theta(\alpha)\cap B(0, 1-\delta))]
&\le \E\left[\sum_{U\in\LC_m(\alpha^-, \alpha^+)}|\diam(U)|^s\right]\\
&\le \sum_{n\ge m} 
\exp(2n\beta)\times \exp(-sn\beta)\times \exp\left(-n\left(
\inf_{\alpha^-\le u\le\alpha^+}\Lambda_{\kappa}^*(u)+o(1)\right)\right)
\tag{By (\ref{eqn::large_deviation})}\\
&=\sum_{n\ge m}\exp \left(n \left(2\beta-s\beta-\inf_{\alpha^-\le u\le \alpha^+}\Lambda_{\kappa}^*(u)+o(1)\right)\right)
\end{align*}
If $s>2-\inf_{\alpha^-\le u\le \alpha^+}\Lambda_{\kappa}^*(u)/\beta$, then (taking $m\to\infty$) we have $\E[\LH_s(\Theta(\alpha)\cap B(0, 1-\delta))]=0$. This implies that 
\[2-\inf_{\alpha^-\le u\le \alpha^+}\Lambda_{\kappa}^*(u)/\beta\ge \dim(\Theta(\alpha)),\quad \text{almost surely}.\]
This holds for any $\beta>\alpha^+>\alpha>\alpha^-$, thus by the continuity of $\Lambda_{\kappa}^*$, we have
\[2-\Lambda_{\kappa}^*(\alpha)/\alpha\ge \dim(\Theta(\alpha)),\quad \text{almost surely}.\]

Finally, when $\alpha<\alpha_{min}$, we see that $\LH_0(\Theta(\alpha)\cap B(0,1-\delta))=0$ almost surely. This implies that $\Theta(\alpha)=\emptyset$ almost surely. 
\end{proof}

\section{Convergence to a limiting field}\label{s.field}

\subsection{Estimates on the disconnection time}\label{sec::disconnecting_time}

The following lemma is the key result of this section:

\begin{lemma}\label{lem::key_lemma}
Fix $\kappa\in [0,4)$. Let $(D_t^0, t\ge 0)$ be $\CGE_\kappa$ targeted at the origin. Then there exist constants $r_0\in (0,1)$, $p_0\in (0,1)$ and $t_0>0$ such that 
\[\PP\left[D_{t_0}^0\subset B(0, r_0)\right]\ge p_0.\]
\end{lemma}

%
%
\begin{proof}
Since the closure of $D_t^0$ is a compact subset of the unit disc, it is sufficient to show that there exist constants $p_0\in (0,1)$ and $t_0>0$ such that 
\begin{equation}\label{eqn::near_conclusion}
\PP[\partial\U\cap \partial D_{t_0}^0=\emptyset]\ge p_0.
\end{equation}

First, we argue that there exist $u, r, \delta>0$ such that for any arc $I\subset \partial\U$ with length less than $\delta$, we have 
\begin{equation}\label{eqn::def_delta_r}
\PP[I\cap \partial D_u^0=\emptyset]\ge r.
\end{equation}
Let $J$ be the collection of positive-length arcs of $\partial\U$ with both endpoints in $\{e^{i\theta}: \theta\in\QQ\}$. Fix any $u>0$; since $\partial D_u^0\cap\partial\U$ is a compact proper subset of $\partial\U$, we know that $\partial\U\setminus \partial D_u^0$ is a union of open arcs, thus 
\[\sum_{I\in J}\PP[I\cap \partial D_u^0=\emptyset]>0.\]
Thus there exists $I_0\in J$ such that $\delta:=|I_0|>0$ and 
\[r:=\PP[I_0\cap \partial D_u^0=\emptyset]>0.\]
Since $D_u^0$ is rotation invariant, we obtain (\ref{eqn::def_delta_r}).

For $\eps>0$, define $E(\eps)$ to be the collection of excursions $\gamma$ in $\U$ with the following property: if $x,y\in\partial\U$ are the two endpoints of $\gamma$, we require that the arc-length from $x$ to $y$ be less than $\eps$ and that $\gamma$ disconnect the origin from the arc from $y$ to $x$. Denote by $E(\gamma)$ the event that $\gamma$ has this property. A standard SLE calculation (see, for instance, \cite{SchrammPercolationFormula}) shows that there is a universal constant $C<\infty$ such that
\begin{align*}
q(\eps):=\mu[E(\eps)]&=\int\int_{|x-y|\le\eps} dxdy H_{\U}(x,y)\mu_{\U, \kappa}^{\#}(x,y)[E(\gamma)]\\
&\le C\int\int_{|x-y|\le\eps} dxdy |x-y|^{-2}|x-y|^{8/\kappa-1}\le C \eps^{8/\kappa-2}.
\end{align*}
In particular, $q(\eps)\to 0$ as $\eps\to 0$. Hence, with $u,r,\delta$ fixed above,  
we can choose $\eps_0\in (0,\delta/2)$ such that 
\begin{equation}\label{eqn::def_eps0}
e^{-uq(\eps_0)}\ge 1-r/2.
\end{equation}

Now let $(\gamma_t, t\ge 0)$ be a PPP with intensity $\mu_{\U, \kappa}$ and let $(D_t^0, t\ge 0)$ be the corresponding growth process targeted at the origin. Let $f_t, \Psi_t$ be the conformal maps defined in Section \ref{sec::growthprocess_construction}. 
Let $T=\inf\{t: \gamma_t\in E(\eps_0)\}$. We know that $T$ has exponential law with parameter $q(\eps_0)$. Fix some arc $I$ with length $\delta$. Conditioned on the set $(\gamma_s, s<T)$ and on the event $E_1=\{I\cap \partial D_T^0=\emptyset\}$, let $I_T$ be the connected component of $\partial D_T^0\setminus \partial \U$ that disconnects $I$ from the origin; see Figure~\ref{fig::explanation_keylemma}(a). Recall that $\Psi_T$ is the conformal map from $D_T^0$ onto $\U$ normalized at the origin. Consider the event $E_2$ that the two endpoints of $\gamma_T$ fall in $\Psi_T(I_T)$. Since $|\Psi_T(I_T)|\ge \delta$ and $\gamma_T\in E(\eps_0)$, we know that the probability of $E_2$ is at least $\delta/(4\pi)$. Conditioned on $(\gamma_s, s\le T)$ and on the event $E_1\cap E_2$, denoting $f_T\circ \Psi_T$ by $\Psi_{T+}$, we have that $\Psi_{T+}^{-1}(\U)\cap \partial
\U=\emptyset$; see Figure \ref{fig::explanation_keylemma}(b). Therefore, for $t>T$,
\[\PP\big[\partial\U\cap \partial D_t^0=\emptyset \bcond \sigma(\gamma_s, s<T), E_1\big]\ge \delta/(4\pi).\]
Thus
\[\PP[\partial\U\cap \partial D_t^0=\emptyset]\ge \PP[t>T, I\cap \partial D_T^0=\emptyset]\times \delta/(4\pi). \]
In order to show (\ref{eqn::near_conclusion}), we need to estimate $\PP[t>T, I\cap \partial D_T^0=\emptyset]$. We have
\begin{align*}
\PP[t>T, I\cap \partial D_T^0=\emptyset]&\ge \PP[t>T>u, I\cap \partial D_u^0=\emptyset]\\
&\ge \PP[t>T>u]-(1-r)\tag{By (\ref{eqn::def_delta_r})}\\
&=e^{-uq(\eps_0)}-e^{-tq(\eps_0)}-(1-r)\\
&\ge r/2-e^{-tq(\eps_0)}\tag{By (\ref{eqn::def_eps0})}.
\end{align*}
We choose $t_0>u$ large so that $e^{-t_0q(\eps_0)}\le r/4$. Then we have 
\[\PP[\partial \U\cap \partial D_{t_0}^0=\emptyset]\ge r\delta/(16\pi),\]
as desired.
\end{proof}

\begin{figure}[ht!]
\begin{subfigure}[b]{0.3\textwidth}
\begin{center}
\includegraphics[width=\textwidth]{\figdir/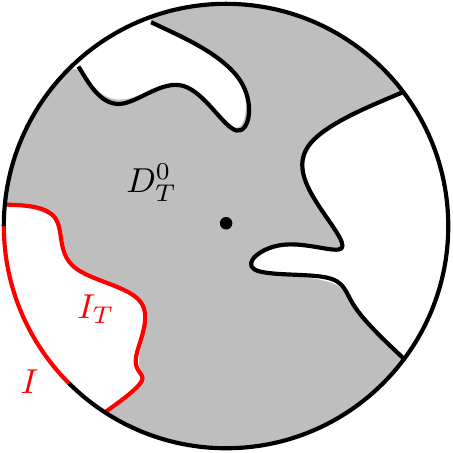}
\end{center}
\caption{Suppose that $I\cap\partial D_T^0=\emptyset$, and let $I_T$ be the connected component of $\partial D_T^0\setminus \partial \U$ that disconnects $I$ from the origin. }
\end{subfigure}
$\quad$
\begin{subfigure}[b]{0.67\textwidth}
\begin{center}
\includegraphics[width=\textwidth]{\figdir/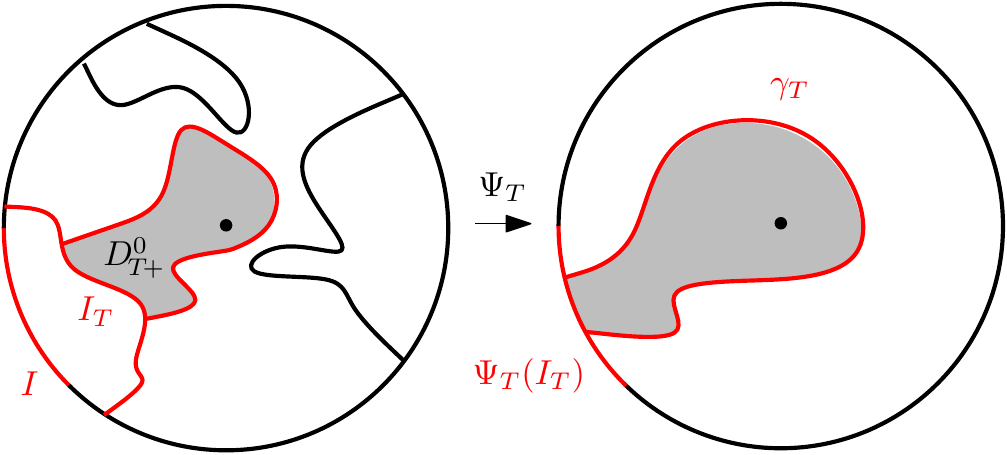}
\end{center}
\caption{Since $|I|\ge\delta$, the harmonic measure of $I_T$ in $D_T^0$ seen from the origin is at least $\delta/2\pi$, thus the arc $\Psi_T(I_T)$ has length at least $\delta$. Note that the distance between the two end points of $\gamma_T$ is less than $\eps_0\le \delta/2$. 
If the two end points of $\gamma_T$ fall in $\Psi_T(I_T)$, then $D_{T+}^0$ will be disjoint of $\partial\U$.}
\end{subfigure}
\caption{\label{fig::explanation_keylemma} }
\end{figure}

\begin{lemma} \label{lem::exp_decay_aux1}
Assume the same notation as in Lemma \ref{lem::key_lemma}. Then there exist constants $c, C\in (0,\infty)$ such that, for all $t>0$,
\[\PP \left[D_t^0\not\subset B(0,r_0)\right]\le C e^{-ct}.\]
\end{lemma}

\begin{proof}
It is sufficient to prove that, for all $n\ge 1$, we have
\begin{equation}\label{eqn::exp_decay_aux}
\PP \left[D_{nt_0}^0\not\subset B(0,r_0)\right]\le (1-p_0)^n.
\end{equation}
We will prove (\ref{eqn::exp_decay_aux}) by induction on $n$. Assume that (\ref{eqn::exp_decay_aux}) holds for $n$. Then, for $n+1$, we have
$$
\PP \left[D_{(n+1)t_0}^0\not\subset B(0,r_0)\right]\le (1-p_0)^n\times \PB{ D_{(n+1)t_0}^0\not\subset B(0,r_0) \md D_{nt_0}^0\not\subset B(0,r_0)}.
$$
Let $\Psi$ be the conformal map from $D_{nt_0}^0$ onto $\U$ normalized at the origin. Since $|\Psi(z)|\ge |z|$, we have 
\begin{equation}\label{eqn::exp_decay_aux2}
\Psi(B(0,r_0))\supset B(0,r_0).
\end{equation}
Thus 
\begin{align*}
\PP \left[D_{(n+1)t_0}^0\not\subset B(0,r_0)\right]&\le (1-p_0)^n\times \PB{ D_{(n+1)t_0}^0\not\subset B(0,r_0) \md D_{nt_0}^0\not\subset B(0,r_0) } \\
&= (1-p_0)^n\times\PP\left[\Psi\left(D_{(n+1)t_0}^0\right)\not\subset \Psi(B(0,r_0))\right]\\
&\le (1-p_0)^n\times\PP\left[D_{t_0}^0\not\subset B(0,r_0)\right]\tag{by $\Psi\left(D_{(n+1)t_0}^0\right)\overset{d}{=}D^0_{t_0}$ and (\ref{eqn::exp_decay_aux2})}\\
&\le (1-p_0)^{n+1},
\end{align*} 
as desired.
\end{proof}

\begin{lemma}\label{lem::exp_decay_aux2}
Assume the same notation as in Lemma \ref{lem::key_lemma}. Then there exist constants $c, C\in (0,\infty)$ such that, for $r>0$ and for all $t>0$,
\[\PP \left[D_t^0\not\subset B(0,r)\right]\le C e^{-ct/\log(1/r)}\times \log(1/r).\]
In particular, this implies that 
\[\PP[T(0,r)>t]\le C e^{-ct/\log(1/r)}\times \log(1/r).\]
\end{lemma}

\begin{proof}
This will be somewhat similar to the previous lemma. First, it is enough to prove that, for any $n\in\N$,
\begin{equation}\label{eqn::exp_decay_r0k}
\P\left[ D_{t}^0 \subset B(0,r_0^n)\right] \ge  \P\left[D_{t/n}^0 \subset B(0,r_0)\right]^n\,,
\end{equation}
because then
\[
\P\left[ D_{t}^0 \not\subset B(0,r_0^n)\right] \le n\, \P\left[D_{t/n}^0 \not\subset B(0,r_0)\right]\,,
\]
and choosing $n=\lceil\log r / \log r_0\rceil$ and using Lemma~\ref{lem::exp_decay_aux1}, we get the upper bound
\[
\PP \left[D_t^0\not\subset B(0,r)\right] \le n \, C e^{-c(t/n)},
\]
which implies the conclusion.

We prove (\ref{eqn::exp_decay_r0k}) by induction on $n$. More precisely, we claim that, for any $k\in\N$ and $u>0$,
\begin{equation}\label{eqn::exp_decay_r0i}
\Pb{ D_{(k+1)u}^0 \subset B(0,r_0^{k+1}) \cond D_{ku}^0 \subset B(0,r_0^k) } \geq \Pb{D^0_u \subset B(0,r_0)}\,,
\end{equation}
and then (\ref{eqn::exp_decay_r0k}) follows by taking $u=t/n$ and a telescoping product for $k=0,1,\dots,n-1$.

Let $\Psi$ be the conformal map from $D_{ku}^0$ onto $\U$ normalized at the origin. We know that $\Psi(D_{(k+1)u}^0)$ has the same law as $D_u^0$. 
To prove (\ref{eqn::exp_decay_r0i}), it is then sufficient to show that, conditioned on $\{D_{ku}^0 \subset B(0,r_0^k)\}$,
\begin{equation}\label{eqn::exp_decay_r_aux}
\Psi(B(0,r_0^{k+1}))\supset B(0,r_0).
\end{equation}
On the event $\{D_{ku}^0\subset B(0, r_0^k)\}$, 
let $\phi_1$ be the conformal map from $D_{ku}^0$ onto $B(0, r_0^k)$ normalized at the origin; then $|\phi_1(z)|\ge |z|$. Let $\phi_2(z)=z/r_0^k$; then $\Psi=\phi_2\circ\phi_1$. Thus 
\[|\Psi(z)|\ge |z|/r_0^k,\]
which implies (\ref{eqn::exp_decay_r_aux}) and hence completes  the proof.
\end{proof}

\begin{corollary} \label{cor::growthprocess_transient}
Assume the same notation as in Lemma \ref{lem::key_lemma}. Then almost surely the growth process $(D_t^0, t\ge 0)$ is transient, i.e., the diameter of $D_t^0$ goes to zero as $t\to\infty$ almost surely.
\end{corollary}

\begin{proof}
For $n\ge 1$, set $r_n=e^{-\sqrt{n}}$. By Lemma \ref{lem::exp_decay_aux2}, we have that 
\[\PP[D_n^0\not\subset B(0,r_n)]\le C\sqrt{n}e^{-c\sqrt{n}}.\]
Thus 
\[\sum_{n}\PP[D_n^0\not\subset B(0,r_n)]<\infty.\]
By the Borel-Cantelli lemma, almost surely there is $N$ such that 
\[D_n^0\subset B(0,r_n),\quad \forall n\ge N.\]
This implies the conclusion.
\end{proof}

\begin{proof}[Proof of Theorem \ref{thm::growthprocess_dimension}, Upper bound]
For $n\ge 1$, define 
\[\Gamma_n=\{z\in\Gamma: T(0,z)\le n\}.\]
By Corollary \ref{cor::growthprocess_transient}, we see that $\Gamma=\cup_{n}\Gamma_n$, thus it is sufficient to show that, for $n\ge 1$, almost surely,
\begin{equation}\label{eqn::dimension_aux}
\dim(\Gamma_n)\le 1+\kappa/8.
\end{equation}
For $m\ge 1$, let $\LU_m$ be the collection of open balls with centers in $e^{-m}\Z^2\cap\U$ and radius $e^{-m}$. Denote by $z(U)$ the center of $U\in\LU_m$. For any $U\in\LU_m$, suppose that $U\cap\Gamma_n\neq\emptyset$ and denote $z(U)$ by $z$; we will argue that this implies 
\begin{equation}\label{eqn::upperbound_aux}
\inrad\left(z; D^{z}_n\right)\le e^{-m}.
\end{equation}
There are two cases: $T(0,z)\le n$ or $T(0,z)>n$. If $T(0,z)\le n$, then $U\cap\Gamma_n\neq\emptyset$ implies $\inrad(z; D^z_{T(0,z)})\le e^{-m}$ which implies (\ref{eqn::upperbound_aux}) since $T(0,z)\le n$. If $T(0,z)>n$, then we know that $D^z_n=D^0_n$. Take $w\in U\cap\Gamma_n$. Since $T(0,w)\le n$, we know that $w\not\in D^0_n$, combining with $|z-w|<e^{-m}$, we obtain (\ref{eqn::upperbound_aux}). Therefore, we have, for any $\lambda\in (0,1-\kappa/8)$
\begin{align*}
\PP\left[U\cap\Gamma_n\neq\emptyset\right]&\le \PP\left[\inrad\left(z; D^{z}_n\right)\le e^{-m}\right]\tag{$z=z(U)$}\\
&\le \PP\left[\CR(z; D^z_n)\le 4e^{-m}\right]\tag{By (\ref{eqn::koebe_quarter})}\\
&=\PP\left[\CR(z; D^z_n)^{-\lambda}\ge (4e^{-m})^{-\lambda}\right]\\
&\le (4e^{-m})^{\lambda}\E\left[\CR(z; D^z_n)^{-\lambda}\right]\\
&=4^{\lambda}e^{-m\lambda}\exp(n\Lambda_{\kappa}(\lambda)).
\end{align*}
We use $\{U\in\LU_m: U\cap \Gamma_n\neq\emptyset\}$ to cover $\Gamma_n$ and to bound $s$-Hausdorff measure of $\Gamma_n$: there is a constant $C$ (only depending on $\kappa,\lambda, n$) such that
\[\E[\LH_s(\Gamma_n)]\le \sum_{U\in\LU_m} \diam(U)^s\PP\left[U\cap\Gamma_n\neq\emptyset\right]\le C e^{2m-ms-m\lambda}.\]
If $s>2-\lambda$, taking $m\to\infty$, we have $\E[\LH_s(\Gamma_n)]=0$, this gives 
\[2-\lambda\ge \dim(\Gamma_n), \quad \text{almost surely}.\]
This holds for any $\lambda\in (0,1-\kappa/8)$, thus 
\[1+\kappa/8\ge \dim(\Gamma_n), \quad \text{almost surely}.\]
\end{proof}

\begin{proof}[Proof of Theorem \ref{thm::growthprocess_dimension}, Lower bound]
Since $\Gamma$ contains the conformal image of entire $\SLE_\kappa$ arcs, we just need to show that such a conformal map cannot have such a bad distortion that would ruin the dimension $1+\kappa/8$ proved in \cite{BeffaraDimension} for $\SLE_{\kappa}$ in the upper half plane. 

Suppose that $(\gamma_t, t\ge 0)$ is a PPP of SLE excursions and $(D_t^0, t\ge 0)$ is the corresponding growth process targeted at the origin. Let $t>0$ be a time when $\gamma_t$ is non-empty, and $\phi$ be the conformal map from $\U$ onto $D^0_t$ normalized at the origin. For any $r<1$, we have some $M_r<\infty$ such that $1/M_r < |\phi'(z)| < M_r$ for all $|z|\leq r$. This implies that the diameter of every subset $U$ of the closed ball $\overline{B(0,r)}$ is changed by at most some finite factor $\tilde M_r$, which implies that 
\begin{equation}\label{eqn::dimdist}
\dim(\phi(\gamma_t \cap \overline{B(0,r)}) = \dim(\gamma_t \cap \overline{B(0,r)}).
\end{equation} 
Since we are dealing with $\kappa<4$ only, the countable union of $\gamma_t \cap \overline{B(0,1-1/n)}$ is all of $\gamma_t$ except for its two endpoints. Thus,~(\ref{eqn::dimdist}) implies that $\dim(\Gamma) \geq \dim\phi(\gamma_t) = 1+\kappa/8$.
\end{proof}

\begin{proof}[Proof of Theorem \ref{thm::disconnection_green}]
By the conformal invariance of $\CGE_\kappa$, it is equivalent to show that, for all $x>0$,
\[\left|\E[T(0,e^{-x})]-\frac{x}{\Lambda'_{\kappa}(0)}\right|\le C.\]
Let $(D_t^0, t\ge 0)$ be the growth process targeted at the origin. 
Define $X(t)=-\log \CR(D_t^0)$, which is the same as the accumulated capacity studied in Section \ref{sec::accumulated_capacity_laplace}.  Define 
\[\tau_x=\inf\{t: X(t) >x\}, \quad Y_x=X(\tau_x).\] It is clear that, for $\lambda<1-\kappa/8$, the process 
\[M_t=\exp\left(\lambda X_t- t\Lambda_{\kappa}(\lambda)\right)\]
is a martingale. 

First, we argue that $(M_{t\wedge \tau_x})_{t\ge 0}$ is a uniformly integrable martingale. Pick $\beta>1$ such that $\lambda\beta<1-\kappa/8$. It is sufficient to show that $(M_{t\wedge \tau_x})_{t\ge 0}$ is uniformly bounded in $L^{\beta}$. We have 
\[\E\left[M_{t\wedge \tau_x}^{\beta}\right]=\E[\exp(\lambda\beta X_{t\wedge \tau_x}-(t\wedge \tau_x)\beta\Lambda_{\kappa}(\lambda))]\le \exp(\lambda\beta x)\E[\exp(\lambda\beta(Y_x-x))].\]
By Propositions \ref{prop::overshoot} and \ref{prop::accumulated_capacity_laplace_exponent}, we know that $\E[\exp(\lambda\beta(Y_x-x))]$ is finite, thus 
\[\sup_t \E\left[M_{t\wedge \tau_x}^{\beta}\right]<\infty,\]
as desired.

Second, we show that $|\E[\tau_x]-x/\Lambda'_{\kappa}(0)|\le C$ for some $C<\infty$ only depending on $\kappa$. Since $(M_{t\wedge \tau_x})_{t\ge 0}$ is a uniformly integrable martingale, we can apply Optional Stopping Theorem and obtain
\begin{equation}\label{eqn::disconnecting_expectation_aux1}
1=\E\left[\exp\left(\lambda Y_x-\Lambda_{\kappa}(\lambda)\tau_x\right)\right].
\end{equation}
Differentiating (\ref{eqn::disconnecting_expectation_aux1}) with respect to $\lambda$ and setting $\lambda=0$, we have 
\[\E[\tau_x]=x/\Lambda'_{\kappa}(0)+\E[(Y_x-x)/\Lambda'_{\kappa}(0)].\]
By Propositions \ref{prop::overshoot} and \ref{prop::accumulated_capacity_laplace_exponent} again, we see that $\E[(Y_x-x)/\Lambda'_{\kappa}(0)]$ is uniformly bounded as desired.

Third, we argue that $T(0, e^{-x})-\tau_x$ has exponentially decaying tail. 
Let $\Psi$ be the conformal map from $D^0_{\tau_x}$ onto $\U$ normalized at the origin. Note that $(\Psi(D^0_{\tau_x+t}), t\ge 0)$ has the same law as $(D^0_t, t\ge 0)$ and is independent of $(D^0_{s}, s\le \tau_x)$. Let $\tilde{T}$ be an independent disconnection time, then, given $D^0_{\tau_x}$
\[\PP[T(0, e^{-x})\ge \tau_x+t]=\PP\left[\tilde{T}\left(0, \Psi(e^{-x})\right)\ge t\right].\] By the Growth Theorem (\cite[Theorem 3.23]{LawlerConformallyInvariantProcesses}), we have that, for any $z\in D^0_{\tau_x}$, 
\[|z|\Psi'(0)\le \frac{|\Psi(z)|}{(1-|\Psi(z)|)^2}.\]
In particular, on the event $\{T(0, e^{-x})>\tau_x\}$, we know that $e^{-x}$ is still contained in $D^0_{\tau_x}$, and since $\Psi'(0)=\exp(Y_x)\ge e^x$, we have 
\[1\le \frac{|\Psi(e^{-x})|}{(1-|\Psi(e^{-x})|)^2},\]
thus
\[|\Psi(e^{-x})|\ge (3-\sqrt{5})/2.\]
Combining with Lemma \ref{lem::exp_decay_aux2}, we have that, for some constants $c, C$
\[\PP[T(0, e^{-x})\ge \tau_x+t]=\PP\left[\tilde{T}\left(0, \Psi(e^{-x})\right)\ge t\right]\le C e^{-ct},\]
as desired.

Finally, we can complete the proof by noting that $T(0, e^{-x})\ge \tau_{x-\log 4}$ by (\ref{eqn::koebe_quarter}) and that $\tau_x-\tau_{x-\log 4}$ has exponentially decaying tail.
\end{proof}

\subsection{Proof of Theorem \ref{thm::distribution_cvg}}

\begin{lemma}\label{lem::distribution_cvg_aux}
Assume the same notation as in Theorem \ref{thm::distribution_cvg}. Then, almost surely,  the sequence $h_n$ converges to some $h$ in $H^{-2-\delta}_{loc}(\U)$.
\end{lemma}
\begin{proof}
To show the conclusion, it is sufficient to check the two conditions in Proposition \ref{prop::distribution_cvg_criterion}. Fix a compact $K\subset \U$. By the conformal invariance of $\CGE_\kappa$, we have $\E[h_n(z)^2]=\E[h_n(0)^2]$. Let $(\sigma_t, t\ge 0)$ be a Poisson point process with intensity $\nu$, then $h_n(0)$ has the same law as $\sum_{s<n}\sigma_s$, and by Compensation Formula for Poisson point process (see \cite[Section 0.5]{BertoinLevyProcesses}), we have
\[\E[h_n(z)^2]=\E[h_n(0)^2]=n.\]
This guarantees the first condition in Proposition \ref{prop::distribution_cvg_criterion}. 

For distinct $z,w\in\U$, let $(D_t^z, t\ge 0)$ and $(D^w_t, t\ge 0)$ be growth processes targeted at $z, w$ respectively and they are coupled so that they are identical up to the disconnection time $T(z,w)$, after which they continue conditionally independently. Given $(D^z_t, t\ge 0)$ and $(D^w_t, t\ge 0)$, 
\begin{align*}
\text{if }T(z,w)\le n, &\text{ we have }\E[(h_{n+1}(z)-h_n(z))(h_{n+1}(w)-h_n(w))]=0;\\
\text{if }T(z,w)> n, &\text{ we have }|\E[(h_{n+1}(z)-h_n(z))(h_{n+1}(w)-h_n(w))]|\le (T(z,w)-n)\wedge 1.
\end{align*}
Thus, we have 
\begin{equation}\label{eqn::distribution_cvg_aux1}
|\E[(h_{n+1}(z)-h_n(z))(h_{n+1}(w)-h_n(w))]|\le \PP[T(z,w)\ge n]. 
\end{equation}

Set $\eps=e^{-\sqrt{n}}$. Suppose $|z-w|\le \eps$. By Theorem \ref{thm::disconnection_green}, we have 
\[\PP[T(z,w)\ge n]\le \frac{1}{n}\E[T(z,w)]\le \frac{1}{n}(G_{\U}(z,w)/\Lambda'_{\kappa}(0)+O(1)).\]
Therefore,
\begin{equation}\label{eqn::distribution_cvg_aux2}
\iint_{K\times K, |z-w|\le\eps}\PP[T(z,w)\ge n]dzdw=O\left(\frac{1}{n}\times \eps^2\log\frac{1}{\eps}\right)=O(e^{-2\sqrt{n}}).
\end{equation}

Suppose $|z-w|\ge \eps$, by Lemma \ref{lem::exp_decay_aux2}, we have 
\[\PP[T(z,w)\ge n]\le C \log(1/\eps)\times e^{-cn/\log(1/\eps)}\le C \sqrt{n} e^{-c\sqrt{n}}.\]
Then 
\begin{equation}\label{eqn::distribution_cvg_aux3}
\iint_{K\times K, |z-w|\ge\eps}\PP[T(z,w)\ge n]dzdw=O(\sqrt{n} e^{-c\sqrt{n}}).
\end{equation}
Combining (\ref{eqn::distribution_cvg_aux1}), (\ref{eqn::distribution_cvg_aux2}) and (\ref{eqn::distribution_cvg_aux3}), we have, for some $c\in (0,2)$, 
\[\iint_{K\times K}|\E[(h_{n+1}(z)-h_n(z))(h_{n+1}(w)-h_n(w))]|dzdw =O(e^{-c\sqrt{n}}).\]
This guarantees the second condition in Proposition \ref{prop::distribution_cvg_criterion}. 
\end{proof}

\begin{proof}[Proof of Theorem \ref{thm::distribution_cvg} Item (1)]
By Lemma \ref{lem::distribution_cvg_aux}, we know that $h_n$ almost surely converges to $h$ in $H^{-2-\delta}_{loc}(\U)$. We will show that $h_t \to h$ in $H^{-2-\delta}_{loc}(\U)$ almost surely. Since $C_c^{\infty}(\U)$ is separable, it is sufficient to show that, for any $f\in C_c^{\infty}(\U)$, we have $\langle h_t, f\rangle\to \langle h, f\rangle$ almost surely. 

For $n\le t\le n+1$, by an argument as in the proof of Lemma \ref{lem::distribution_cvg_aux}, we have, for some $c\in (0,2)$,
\[\iint_{K\times K}|\E[(h_{t}(z)-h_n(z))(h_{t}(w)-h_n(w))]|dzdw =O(e^{-c\sqrt{n}}).\]
This implies that 
\[\E\left[(\langle h_t, f\rangle-\langle h_n, f\rangle)^2\right]=O(e^{-c\sqrt{n}}).\]
Since $(\langle h_t, f\rangle)_{t\ge 0}$ is a martingale, by Doob's maximal inequality, we have 
\[\PP\left[\sup_{n\le t\le n+1}|\langle h_t, f\rangle-\langle h_n, f\rangle|\ge e^{-c\sqrt{n}/4}\right]=O(e^{-c\sqrt{n}/2}).\]
Thus
\[\sum_n \PP\left[\sup_{n\le t\le n+1}|\langle h_t, f\rangle-\langle h_n, f\rangle|\ge e^{-c\sqrt{n}/4}\right]<\infty.\]
By Borel-Cantelli Lemma, we have, almost surely, there exists $m$ such that 
\[\sup_{n\le t\le n+1}|\langle h_t, f\rangle-\langle h_n, f\rangle|\le e^{-c\sqrt{n}/4},\quad \forall n\ge m.\]
Combining with the fact that $\langle h_n, f\rangle\to \langle h, f\rangle$ almost surely, we have $\langle h_t, f\rangle\to \langle h, f\rangle$ almost surely.
\end{proof}

\begin{proof}[Proof of Theorem \ref{thm::distribution_cvg} Item (2)]
We will show that the limiting distribution $h$ is measurable with respect to the $\sigma$-algebra $\Sigma$ generated by $\Gamma$, the collection of $\SLE_{\kappa}$ excursions, and the weights $(\sigma_{\gamma})_{\gamma\in\Gamma}$. 

Since $C_c^{\infty}(\U)$ is separable, there is a countable dense subset $S$ of $C_c^{\infty}(\U)$. Note that the distribution $h_n$ is $\Sigma$-measurable and since $h$ is an almost sure limit of $h_n$, the limiting filed $h$ is determined by the values $\{\langle h_n, f\rangle: n\ge 1, f\in S\}$. Thus $h$ is $\Sigma$-measurable.
\end{proof}

\begin{proof}[Proof of Theorem \ref{thm::distribution_cvg} Item (3)]
The conformal invariance of the limiting distribution $h$ is an immediate consequence of the conformal invariance of $\CGE_\kappa$. We emphasize that the conformal invariance of $\CGE_\kappa$ in Theorem \ref{thm::growing_process_conf_inv} does not require time change.
\end{proof}

\subsection{Extremes of the limiting field}
Assume the same notation as in Theorem \ref{thm::distribution_cvg}. Fix $z\in\U$, 
by Strong Law of Large Numbers and Central Limit Theorem for subordinators, we can derive the typical behavior of the field $h_t(z)$ as $t\to\infty$: we have almost surely 
\[\lim_{t\to\infty}h_t(z)/t=0.\]
We are also interested in the points that $h_t(z)/t$ behaves in an abnormal way. Define 
\[\Phi_{\nu}(x)=\{z\in\U: \lim_{t\to\infty}h_t(z)/t=x-\bar{m}\}.\]
Let $\Lambda_{\nu}$ and $\Lambda_{\nu}^*$ be the Laplace exponent and its Fenchel-Legendre transform corresponding to the weight measure: for $\lambda, x\in\R$
\[\Lambda_{\nu}(\lambda)=\int(e^{\lambda x}-1)\nu[dx],\quad \Lambda_{\nu}^*(x)=\sup_{\lambda\in\R}(\lambda x-\Lambda_{\nu}(\lambda)).\]
We assume that $\Lambda_{\nu}^*$ is continuous.  
Then by Large Deviation again, we can derive the following estimate on the Hausdorff dimension of $\Phi_{\nu}(x)$.
\begin{theorem} Fix $\kappa\in [0,4)$ and recall that $\Theta(\alpha)$ is defined in (\ref{eqn::Theta_alpha}) and $\alpha_{\min}$ is defined in (\ref{eqn::alpha_min}). 
For $\alpha\ge \alpha_{min}$, we have almost surely
\[\begin{cases}
\dim\left(\Theta(\alpha)\cap \Phi_{\nu}(x)\right)\le 2-\Lambda_{\kappa}^*(\alpha)/\alpha-\Lambda_{\nu}^*(x)/\alpha, &\text{if } \Lambda_{\kappa}^*(\alpha)/\alpha+\Lambda_{\nu}^*(x)/\alpha\le 2;\\
\Theta(\alpha)\cap \Phi_{\nu}(x)=\emptyset, &\text{if } \Lambda_{\kappa}^*(\alpha)/\alpha+\Lambda_{\nu}^*(x)/\alpha> 2.
\end{cases}\]
\end{theorem}
\begin{proof}
Define 
\[\Phi_{\nu}^-(x)=\left\{z\in\U: \limsup_{n\to \infty}\frac{h_n(z)}{n}\le x-\bar{m}\right\},\quad \Phi_{\nu}^+(x)=\left\{z\in\U: \liminf_{n\to \infty}\frac{h_n(z)}{n}\ge x-\bar{m}\right\}.\]
Assume that $\alpha, x$ are chosen so that $2-\Lambda_{\kappa}^*(\alpha)/\alpha-\Lambda_{\nu}^*(x)/\alpha\ge 0$. It suffices to show that, almost surely,
\begin{align}
\dim(\Theta(\alpha)\cap \Phi_{\nu}^-(x))\le 2-\Lambda_{\kappa}^*(\alpha)/\alpha-\Lambda_{\nu}^*(x)/\alpha, \quad \text{for }x<\bar{m};\label{eqn::extreme_minus_minus}\\
\dim(\Theta(\alpha)\cap \Phi_{\nu}^+(x))\le 2-\Lambda_{\kappa}^*(\alpha)/\alpha-\Lambda_{\nu}^*(x)/\alpha, \quad \text{for }x>\bar{m};\label{eqn::extreme_minus_plus}
\end{align}
We will show (\ref{eqn::extreme_minus_minus}) and the bound in (\ref{eqn::extreme_minus_plus}) can be proved similarly.

We assume the same notation as in the proof of Theorem \ref{thm::abnormal_decay_upper}. Define 
\[\LU_n(u^-, u^+)=\left\{U\in\LU_n: u^-\le \frac{-\log \CR\left(z(U); D_n^{z(U)}\right)}{n}\le u^+\right\},\quad \LV_n^-(v)=\left\{U\in\LU_n: \frac{h_n(z(U))}{n}\le v-\bar{m}\right\}.\]

For $u^-<u^+$ and $v<\bar{m}$, by Cram\'er's Theorem, we have
\[\PP\left[U\in \LU_n(u^-, u^+)\cap \LV_n^-(v)\right]\le \exp\left(-n\left(\inf_{u^-\le u\le u^+}\Lambda_{\kappa}^*(u)+\Lambda_{\nu}^*(v)+o(1)\right)\right),\]
where the $o(1)$ term tends to zero as $n\to \infty$ uniformly in $U$. Define 
\[\LC_m^-(u^-, u^+, v)=\cup_{n\ge m}\LU_n(u^-, u^+)\cap \LV_n^-(v).\]
Pick $\alpha^-<\tilde{\alpha}^-<\alpha<\tilde{\alpha}^+<\alpha^+<\beta$ and $x<x'<\bar{m}$. 
We claim that $\LC_m^-(\alpha^-, \alpha^+, x')$ is a cover for $\Theta(\alpha)\cap \Phi_{\nu}^-(x)\cap B(0,1-\delta)$.
For any $z\in \Theta(\alpha)\cap \Phi_{\nu}^-(x)\cap B(0,1-\delta)$, since 
\[\lim_{n\to \infty}\frac{-\log\CR(z; D_n^z)}{n}=\alpha,\quad \limsup_{n\to \infty}\frac{h_n(z)}{n}\le x-\bar{m},\]
we know that, for $n$ large enough,
\[\exp(-n\tilde{\alpha}^-)\ge\CR(z; D_n^z)\ge \exp(-n\tilde{\alpha}^+),\quad h_n(z)\le n(x'-\bar{m}).\]
Let $w$ be the point in $e^{-n\beta}\Z^2$ that is closest to $z$ and denote by $U\in\LU_n$ the ball with center $w$, then $w$ is contained in $D^z_n$. Moreover, 
\[\exp(-n\alpha^-)\ge \CR(w;D^w_n)\ge \exp(-n\alpha^+),\quad h_n(w)=h_n(z)\le n (x'-\bar{m}).\]
Thus $z\in U\in \LC_m^-(\alpha^-, \alpha^+, x')$. This implies that $\LC_m^-(\alpha^-, \alpha^+, x')$ is a cover for $\Theta(\alpha)\cap \Phi_{\nu}^-(x)\cap B(0,1-\delta)$. We use these covers to bound $s$-Hausdorff measure of $\Theta(\alpha)\cap \Phi_{\nu}^-(x)\cap B(0,1-\delta)$. 
\[\E[\LH_s(\Theta(\alpha)\cap \Phi_{\nu}^-(x)\cap B(0,1-\delta))]\le \sum_{n\ge m}\exp\left(n\left(2\beta-s\beta-\inf_{\alpha^-\le u\le \alpha^+}\Lambda_{\kappa}^*(u)-\Lambda_{\nu}^*(x')+o(1)\right)\right).\] This implies that 
\[2-\inf_{\alpha^-\le u\le \alpha^+}\Lambda_{\kappa}^*(u)/\beta-\Lambda_{\nu}^*(x')/\beta\ge \dim(\Theta(\alpha)\cap \Phi_{\nu}^-(x)), \quad \text{almost surely}.\]
This holds for any $\beta>\alpha^+>\alpha>\alpha^-$ and $x'>x$, thus by the continuity of $\Lambda_{\kappa}^*$ and $\Lambda_{\nu}^*$, we have 
\[2-\Lambda_{\kappa}^*(\alpha)/\alpha-\Lambda_{\nu}^*(x)/\alpha\ge \dim(\Theta(\alpha)\cap \Phi_{\nu}^-(x)), \quad \text{almost surely}.\]
This completes the proof for (\ref{eqn::extreme_minus_minus}).
\end{proof}

\section{Open questions}\label{s.open}

Even though we know that the Hausdorff dimension of the closure of $\cup_{t\ge 0} \partial D^0_t$ is $1+\kappa/8$, the dimension of a single $\partial D^0_t$ could be smaller; intuitively, this happens if the growing arcs form bottlenecks, producing shortcuts in $\partial D^0_t$. However, we do not expect this to happen, which might be possible to prove by arguments similar to the proofs of Lemma~\ref{lem::key_lemma} and Theorem~\ref{thm::growthprocess_dimension}:

\begin{question}[Dimension of the boundary]
Is the Hausdorff dimension of the closure of $\partial D^0_t$ almost surely $1+\kappa/8$?
\end{question}

%

One can view $\partial D^0_t$ as a Markov process on loops surrounding the origin. What is its stationary measure?

\begin{question}[Stationary loop]
Consider the rescaled loop $L^0_t:=\exp(t\Lambda'_\kappa(0))\, \partial D^0_t$ around the origin. Show that it has a limiting distribution as $t\to\infty$, and identify this law.
\end{question}

Finally, possibly the most interesting question:

\begin{question}[Discrete models] Identify the growth process $\CGE_\kappa$ for some values of $\kappa$ as the scaling limit of some discrete models.
\end{question}

\bibliographystyle{alpha}
\bibliography{hao_wu_thesis,HastingsLevitov}

\begin{thebibliography}{MWW13}

\bibitem[AK08]{AlbertsKozdronIntersectionProbaSLEBoundary}
Tom Alberts and Michael~J Kozdron.
\newblock Intersection probabilities for a chordal {SLE} path and a semicircle.
\newblock {\em Electron. Commun. Probab}, 13:448--460, 2008.

\bibitem[Bat53]{emot}
Harry Bateman.
\newblock Higher transcendental functions [volume i].
\newblock 1953.

\bibitem[Bef08]{BeffaraDimension}
Vincent Beffara.
\newblock The dimension of the {SLE} curves.
\newblock {\em Ann. Probab.}, 36(4):1421--1452, 2008.

\bibitem[Ber96]{BertoinLevyProcesses}
Jean Bertoin.
\newblock {\em L\'evy processes}, volume 121 of {\em Cambridge Tracts in
  Mathematics}.
\newblock Cambridge University Press, Cambridge, 1996.

\bibitem[Ber99]{BertoinSubordinators}
Jean Bertoin.
\newblock Subordinators: examples and applications.
\newblock In {\em Lectures on probability theory and statistics
  ({S}aint-{F}lour, 1997)}, volume 1717 of {\em Lecture Notes in Math.}, pages
  1--91. Springer, Berlin, 1999.

\bibitem[CM01]{carleson2001aggregation}
Lennart Carleson and Nikolai Makarov.
\newblock Aggregation in the plane and {L}oewner's equation.
\newblock {\em Communications in Mathematical Physics}, 216(3):583--607, 2001.

\bibitem[CN06]{CamiaNewmanPercolationFull}
Federico Camia and Charles~M. Newman.
\newblock Two-dimensional critical percolation: the full scaling limit.
\newblock {\em Comm. Math. Phys.}, 268(1):1--38, 2006.

\bibitem[DZ10]{DemboZeitouniLargeDeviations}
Amir Dembo and Ofer Zeitouni.
\newblock {\em Large deviations techniques and applications}, volume~38 of {\em
  Stochastic Modelling and Applied Probability}.
\newblock Springer-Verlag, Berlin, 2010.
\newblock Corrected reprint of the second (1998) edition.

\bibitem[Hal00]{halsey2000diffusion}
Thomas~C Halsey.
\newblock Diffusion-limited aggregation: a model for pattern formation.
\newblock {\em Physics Today}, 53(11):36--41, 2000.

\bibitem[HL98]{hastings1998laplacian}
Matthew~B Hastings and Leonid~S Levitov.
\newblock Laplacian growth as one-dimensional turbulence.
\newblock {\em Physica D: Nonlinear Phenomena}, 116(1):244--252, 1998.

\bibitem[KL07]{Kozdron2007configurational}
Michael~J Kozdron and Gregory~F Lawler.
\newblock The configurational measure on mutually avoiding sle paths.
\newblock {\em Universality and Renormalization: From Stochastic Evolution to
  Renormalization of Quantum Fields}, 50:199--224, 2007.

\bibitem[KS12]{KemppainenSmirnovRandomCurves}
Antti Kemppainen and Stanislav Smirnov.
\newblock Random curves, scaling limits and loewner evolutions.
\newblock {\em arXiv preprint arXiv:1212.6215}, 2012.

\bibitem[Law05]{LawlerConformallyInvariantProcesses}
Gregory~F. Lawler.
\newblock {\em Conformally invariant processes in the plane}, volume 114 of
  {\em Mathematical Surveys and Monographs}.
\newblock American Mathematical Society, Providence, RI, 2005.

\bibitem[LSW03]{LawlerSchrammWernerConformalRestriction}
Gregory~F. Lawler, Oded Schramm, and Wendelin Werner.
\newblock Conformal restriction: the chordal case.
\newblock {\em J. Amer. Math. Soc.}, 16(4):917--955 (electronic), 2003.

\bibitem[LSW04]{LawlerSchrammWernerLERWUST}
Gregory~F. Lawler, Oded Schramm, and Wendelin Werner.
\newblock Conformal invariance of planar loop-erased random walks and uniform
  spanning trees.
\newblock {\em Ann. Probab.}, 32(1B):939--995, 2004.

\bibitem[MWW13]{MillerWatsonWilsonCLENestingfield}
Jason Miller, Samuel~S Watson, and David~B Wilson.
\newblock The conformal loop ensemble nesting field.
\newblock {\em Probability Theory and Related Fields}, pages 1--33, 2013.

\bibitem[NT12]{norris2012hastings}
James Norris and Amanda Turner.
\newblock Hastings--{L}evitov aggregation in the small-particle limit.
\newblock {\em Communications in Mathematical Physics}, 316(3):809--841, 2012.

\bibitem[RS05]{RohdeSchrammSLEBasicProperty}
Steffen Rohde and Oded Schramm.
\newblock Basic properties of {SLE}.
\newblock {\em Ann. of Math. (2)}, 161(2):883--924, 2005.

\bibitem[RZ05]{rohde2005some}
Steffen Rohde and Michel Zinsmeister.
\newblock Some remarks on {L}aplacian growth.
\newblock {\em Topology and its Applications}, 152(1):26--43, 2005.

\bibitem[Sch00]{SchrammScalinglimitsLERWUST}
Oded Schramm.
\newblock Scaling limits of loop-erased random walks and uniform spanning
  trees.
\newblock {\em Israel J. Math.}, 118:221--288, 2000.

\bibitem[Sch01]{SchrammPercolationFormula}
Oded Schramm.
\newblock A percolation formula.
\newblock {\em Electron. Comm. Probab}, 6:115--120, 2001.

\bibitem[Smi01]{SmirnovPercolationConformalInvariance}
Stanislav Smirnov.
\newblock Critical percolation in the plane: conformal invariance, {C}ardy's
  formula, scaling limits.
\newblock {\em C. R. Acad. Sci. Paris S\'er. I Math.}, 333(3):239--244, 2001.

\bibitem[SSW09]{SchrammSheffieldWilsonConformalRadii}
Oded Schramm, Scott Sheffield, and David~B. Wilson.
\newblock Conformal radii for conformal loop ensembles.
\newblock {\em Comm. Math. Phys.}, 288(1):43--53, 2009.

\bibitem[SW05]{SchrammWilsonSLECoordinatechanges}
Oded Schramm and David~B. Wilson.
\newblock S{LE} coordinate changes.
\newblock {\em New York J. Math.}, 11:659--669 (electronic), 2005.

\bibitem[SW12]{SheffieldWernerCLE}
Scott Sheffield and Wendelin Werner.
\newblock Conformal loop ensembles: the {M}arkovian characterization and the
  loop-soup construction.
\newblock {\em Ann. of Math. (2)}, 176(3):1827--1917, 2012.

\bibitem[Tao10]{TaoEpsilon}
Terrence Tao.
\newblock An epsilon of room, {V}ol.~i.
\newblock {\em American Mathematical Society}, 2010.

\bibitem[VST15]{viklund2015small}
Fredrik~Johansson Viklund, Alan Sola, and Amanda Turner.
\newblock Small-particle limits in a regularized {L}aplacian random growth
  model.
\newblock {\em Communications in Mathematical Physics}, 334(1):331--366, 2015.

\bibitem[Wer04]{WernerRandomPlanarcurves}
Wendelin Werner.
\newblock Random planar curves and {S}chramm-{L}oewner evolutions.
\newblock In {\em Lectures on probability theory and statistics}, volume 1840
  of {\em Lecture Notes in Math.}, pages 107--195. Springer, Berlin, 2004.

\bibitem[Wer05]{WernerConformalRestrictionRelated}
Wendelin Werner.
\newblock Conformal restriction and related questions.
\newblock {\em Probab. Surv.}, 2:145--190, 2005.

\bibitem[Wer07]{WernerLecturePercolation}
Wendelin Werner.
\newblock Lectures on two-dimensional critical percolation.
\newblock {\em IAS Park City Graduate Summer School}, 2007.

\bibitem[Wil96]{wilson1996generating}
David~Bruce Wilson.
\newblock Generating random spanning trees more quickly than the cover time.
\newblock In {\em Proceedings of the twenty-eighth annual ACM symposium on
  Theory of computing}, pages 296--303. ACM, 1996.

\bibitem[WJS81]{witten1981diffusion}
TA~Witten~Jr and Leonard~M Sander.
\newblock Diffusion-limited aggregation, a kinetic critical phenomenon.
\newblock {\em Physical review letters}, 47(19):1400, 1981.

\bibitem[WW13]{WernerWuCLEExploration}
Wendelin Werner and Hao Wu.
\newblock On conformally invariant {CLE} explorations.
\newblock {\em Comm. Math. Phys.}, 320(3):637--661, 2013.

\end{thebibliography}

\bigbreak 
\noindent {\bf G\'abor Pete}\\
R\'enyi Institute, Hungarian Academy of Sciences, Budapest,\\
and Institute of Mathematics, Budapest University of Technology and Economics\\
\url{http://www.math.bme.hu/~gabor}\\

\medskip
\noindent {\bf Hao Wu}\\
\noindent NCCR/SwissMAP, Section de Math\'{e}matiques\\
\noindent Universit\'{e} de Gen\`{e}ve\\
\noindent Switzerland\\
\noindent {\tt hao.wu.proba@gmail.com}

\end{document}